\documentclass[10pt]{article}
\usepackage{amssymb}
\usepackage{amsmath}
\usepackage{amsthm}
\usepackage{amscd}
\usepackage{hhline}
\usepackage{graphicx}
\usepackage{geometry}
\usepackage[cmtip,all]{xy}
\usepackage{upref}
\usepackage{pgf,tikz}
\usetikzlibrary{arrows}
\definecolor{qqqqff}{rgb}{0,0,1}
\usepackage[breaklinks,pdfstartview=FitH,colorlinks,linktocpage,linkcolor=blue,citecolor=blue,urlcolor=blue]{hyperref}

\newcommand{\Q}{\mathbb{Q}}
\newcommand{\R}{\mathbb{R}}
\newcommand{\Z}{\mathbb{Z}}
\newcommand{\tensor}{\otimes}
\newcommand{\dlra}[1]{\stackrel{#1}{\longrightarrow}}

\newcommand{\Diff}{\text{Diff}}
\newcommand{\wt}{\widetilde}

\newtheorem{problem}{Problem}[section]

\newtheorem{theorem}{Theorem}[section]
\newtheorem{lemma}[theorem]{Lemma}

\newtheorem{corollary}[theorem]{Corollary}
\newtheorem{claim}[theorem]{Claim}

\theoremstyle{definition}

\newtheorem*{definition*}{Definition}
\newtheorem{example}[theorem]{Example}

\theoremstyle{remark}
\newtheorem{remark}[theorem]{Remark}

\numberwithin{equation}{section}

\DeclareMathOperator{\rank}{rk}

\DeclareMathOperator{\kernel}{Ker}

\begin{document}


\title{The rational classification of links\\ of codimension $>2$}

\author{Diarmuid Crowley, Steven C. Ferry and Mikhail Skopenkov}

\date{}
\maketitle









\begin{abstract}




Let $m$ and $p_1, \dots, p_r< m - 2$ be positive integers.
The set of links of codimension $>2$, $E^m(\sqcup_{k=1}^r S^{p_k})$, is the set
%
%
of smooth isotopy classes of smooth embeddings $\sqcup_{k=1}^r S^{p_k}\to S^m$.  Haefliger showed that $E^m(\sqcup_{k=1}^r S^{p_k})$ is a finitely generated abelian group with respect to embedded connected summation and computed its rank in the case of knots, i.e.\,\,$r=1$.   For $r>1$ and for restrictions on $p_1, \dots, p_r$ the rank of this group can be computed using results of Haefliger or Nezhinsky.
Our main result determines the rank of the group $E^m(\sqcup_{k=1}^r S^{p_k})$ in general.  In particular we determine precisely when $E^m(\sqcup_{k=1}^r S^{p_k})$ is finite.  We also accomplish these tasks for framed links.  Our proofs are based on the Haefliger exact sequence for groups of links and the theory of Lie algebras.

\smallskip



\end{abstract}

\maketitle

\footnotetext[0]{The third author was supported in part by INTAS grant 06-1000014-6277, Moebius Contest Foundation for Young Scientists and Euler Foundation.}

\section{Introduction}\label{sec-intro}

Let $m$ and $p_1, \dots, p_r< m - 2$ be positive integers.
The \emph{set of links} in codimension $>2$ is the set
%
\[ E^m(\sqcup_{k=1}^r S^{p_k}) :=
\bigl\{ f  \colon \! \sqcup_{k=1}^r S^{p_k} \hookrightarrow S^m \bigr\}/\text{isotopy}\]
%
of smooth isotopy classes of smooth embeddings  $f \colon \!  \sqcup_{k=1}^r S^{p_k}\hookrightarrow S^m$.
This set is a finitely generated abelian group with respect to componentwise embedded connected summation \cite{Hae62T}. In addition to its intrinsic mathematical interest, the group of links is related to the classification of handlebodies, the mapping class groups of certain manifolds and to the set of embeddings of more general disjoint unions of manifolds: we discuss these applications in Section~\ref{sec-further}.


Up to extension problems, the computation of the group $E^m(\sqcup_{k=1}^r S^{p_k})$ was reduced to problems in unstable homotopy theory in the seminal papers of Levine and Haefliger \cite{Le65, Hae66A, Hae66C}.  However, these problems include the determination of the unstable homotopy groups of spheres.  Hence one expects that the precise computation of $E^m(\sqcup_{k=1}^r S^{p_k})$ is in general extremely difficult.


In this paper we address the following simpler question: \emph{what is the rank of the group $E^m(\sqcup_{k=1}^r S^{p_k})$ and in particular  for which $m,p_1,\dots,p_r$ is the group $E^m(\sqcup_{k=1}^r S^{p_k})$ finite}? This question is in part motivated by analogy to rational homotopy theory and also the rational classification of link maps by Koschorke \cite{Kos90, HaKa98}.

The main results of this paper give an explicit formula for the rank of the group $E^m_{}(\sqcup_{k=1}^r S^{p_k})$ (Theorem~\ref{thm-general}) and also a criterion which determines precisely when the group of $E^m_{}(\sqcup_{k=1}^r S^{p_k})$ is finite (Corollary~\ref{cor-globalfinite}).
We also accomplish the same tasks for the groups of framed links (Section~\ref{ssec-framed}).

\subsection{Background and a useful observation}
An embedding $f\colon \! \sqcup_{k=1}^r S^{p_k}\hookrightarrow S^m$ is called a \emph{link}, its restrictions $f\colon S^{p_k}\hookrightarrow S^m$ are called \emph{components}.
For 
one-component links, or \emph{knots}, the question posed above was answered by Haefliger:

\begin{theorem} \label{th1} \textup{(See \cite[Corollary~6.7]{Hae66A})} Assume
that $p<m-2$. The group $E^m(S^p)$ has rank $0$ or $1$:  it is infinite if and only if $p+1$ is divisible by $4$ and
$p>\frac{2}{3}(m-2)$.
\end{theorem}



Typically one approaches 
multi-component links by 
studying the sub-links which are obtained by deleting one or more components from the original link. 
In particular, a link is called 
\emph{Brunnian} if it becomes trivial after removing any one of its components. An example of a nontrivial Brunnian link is the Borromean rings; cf.~\cite{DT-etal-11}. A link \emph{has unknotted components}, if each of its components is a trivial knot.
Denote by $E^m_{\mathrm{B}}(\sqcup_{k=1}^r S^{p_k})$ and
$E^m_{\mathrm{U}}(\sqcup_{k=1}^r S^{p_k})$ respectively the subgroups of Brunnian links and links having unknotted components: observe that these two subgroups coincide for $r=2$.

\begin{theorem}\label{th-structure} \textup{(See \cite[Sections~2.4 and~9.3]{Hae66C})}
Assume that $p_1,\dots,p_r < m-2$. Then there are isomorphisms
\[ E^m(\sqcup_{k=1}^r S^{p_k}) \,\,
\cong \,\, E^m_{\mathrm{U}}(\sqcup_{k=1}^r S^{p_k})
\oplus\bigoplus_{k=1}^r E^m(S^{p_k}) \,\,
\cong \,\, \bigoplus_{S\subset\{1,\dots,r\},\, S\ne\emptyset} E^m_{\mathrm{B}}(\sqcup_{k\in S} S^{p_k}), \]
where the last sum is over all nonempty subsets $S\subset\{1,\dots,r\}$.
\end{theorem}
%
%
%
%
Under certain dimension restrictions Haefliger and Nezhinsky have found explicit descriptions of the isotopy classes of Brunnian links in terms of homotopy groups of spheres and Stiefel manifolds \cite{Hae66A, Hae66C, Ne82, Sko08P}. 

For spheres of arbitrary dimensions $p_k$ (with each $p_k < m-2)$ Haefliger constructed a 
long exact sequence (see \cite[Theorem~1.3]{Hae66C}, \cite[Theorem~1.1]{Ha86}):
%
%
%
\begin{equation} \label{seq:UHae}
\dots \to \Pi_{m-1}^{(q)} \dlra{\mu} E^m_{\mathrm{U}}(\sqcup_{k=1}^r S^{p_k}) \dlra{\lambda}  \Lambda_{(p)}^{(q)} \dlra{w} \Pi_{m-2}^{(q)} \dlra{\mu}  E^{m-1}_{\mathrm{U}}(\sqcup_{k=1}^r S^{p_k-1}) \to \dots
\end{equation}
Here $\Lambda_{(p)}^{(q)}$ and $\Pi_{m-2}^{(q)}$ are certain finitely generated abelian groups defined via homomorphisms between the homotopy groups of appropriate wedges of spheres, see 
Section~\ref{ssec-Hae} for their definition.  The homomorphisms $\lambda$ and $\mu$ are topologically defined in \cite[Section~1.4]{Hae66C}: we do not explicitly consider them in this paper. Rather we note that up to extension, the group of links is determined by the homomorphism $w$.
Moreover, since $w$ is defined using Whitehead products, the Haefliger sequence~\eqref{seq:UHae} reduces the determination of the group $E_{}(\sqcup_{k=1}^r S^{p_k})$ to a problem in unstable homotopy theory and an extension problem.  In particular determining the rank of $E_{}(\sqcup_{k=1}^r S^{p_k})$ is reduced to a problem in unstable homotopy theory.



The starting point of our investigation is the following simple but crucial observation:

\begin{lemma} \label{lem:Hae-splits}
After tensoring with $\Q$ the Haefliger sequence~\eqref{seq:UHae} splits into the short exact sequences
\[ 0 \to E^m_{\mathrm{U}}(\sqcup_{k=1}^r S^{p_k}) \tensor \Q \stackrel{\lambda \tensor {\rm Id}_\Q} {\xrightarrow{\hspace*{1cm}}} \Lambda_{(p)}^{(q)} \tensor \Q \stackrel{w \tensor {\rm Id}_\Q} {\xrightarrow{\hspace*{1cm}}} \Pi_{m-2}^{(q)} \tensor \Q \to 0.\]
%
%
\end{lemma}

The Haefliger sequence~\eqref{seq:UHae} and
Lemma~\ref{lem:Hae-splits} are the basis for all the results stated in the remainder of the introduction. 
Note that the Haefliger sequence~\eqref{seq:UHae} itself does not split as in Lemma~\ref{lem:Hae-splits} in general; see Lemma~\ref{lem-integralnonsplit} below.


\subsection{Finiteness criteria for the group of links}\label{ssec-fin}

In this subsection we state criteria which determine precisely when the group $E^m_{}(\sqcup_{k=1}^r S^{p_k})$ is finite.  We begin with finiteness criteria for \emph{Brunnian links} and conclude with the general case.






\begin{theorem}\label{thm-1a}
Assume that $r>2$ and $p_1,\dots,p_r<m-2$. Then
the group $E^m_{\mathrm{B}}(\sqcup_{k=1}^r S^{p_k})$
is infinite if and only if the equation $(m-p_1-2)x_1+\dots+(m-p_r-2)x_r=m-3$ has a solution in positive integers $x_1,\dots,x_r$.
\end{theorem}

In the case of $2$ component links the criterion is more complicated. It involves certain \emph{finiteness-checking sets} $FCS (i,j)\subset \mathbb{Z}^2$ which depend only on the parity of $i$ and $j$ which are defined in Table~\ref{formal} below.
A part of each set is drawn in the table; the rest of the set is obtained by the evident periodicity.

\begin{theorem} \label{thm-1b}
Assume that $p_1,p_2< m-2$. Then
the group $E^m_{\mathrm{B}}(\sqcup_{k=1}^2 S^{p_k})$
is infinite if and only if
there exists a point $(x_1,x_2)\in FCS (m-p_1,m-p_2)$ such that $(m-p_1-2)x_1+(m-p_2-2)x_2=m-3$.
\end{theorem}

The finiteness-checking set can be considered as a nomogram: to establish infiniteness one draws the line $(m-p_1-2)x_1+(m-p_2-2)x_2=m-3$ and looks if the line intersects the set $FCS (m-p_1,m-p_2)$.


\begin{table}[h]
\caption{Definition of the finiteness-checking set $FCS (i,j)$}
\label{formal}
\begin{center}
\begin{tabular}{|p{4cm}|p{4cm}|p{4cm}|}
\hline
\multicolumn{3}{|p{12cm}|}{$FCS (i,j)$ is the set of pairs $(x,y)\in \mathbb{Z}\times\mathbb{Z}$ such that $x,y>0$ and at least one of the following conditions holds~---}\\[2pt]
\hline
\multicolumn{1}{|c|}{for $i,j$ even:} & \multicolumn{1}{|c|}{for $i$ odd, $j$ even:} & \multicolumn{1}{|c|}{for $i,j$ odd:}\\[2pt]
\hline
\begin{itemize}
\item $x=1$ and $y=1$;
\item $x=2$ and $2\,|\,y$;
\item $x=3$ and $y=3$;
\item $x=3$ and $y\ge5$;
\item $x\ge4$ and $y\ge4$;
\item $2\,|\,x$ and $y=2$;
\item $x\ge5$ and $y=3$.
\end{itemize}
&
\begin{itemize}
\item $x=1$ and $y=1$;
\item $x=2$ and $2\,|\,y+1$;
\item $x=3$ and $y\ge2$;
\item $x\ge4$ and $y\ge4$;
\item $4\,|\,x$ and $y=2$;
\item $4\,|\,x+1$ and $y=2$;
\item $x\ge5$ and $y=3$.
\end{itemize}
&
\begin{itemize}
\item $x=1$ and $y=1$;
\item $x=2$ and $4\,|\,y+2$;
\item $x=2$ and $4\,|\,y+3$;
\item $x\ge 3$ and $y\ge3$;
\item $4\,|\,x+2$ and $y=2$;
\item $4\,|\,x+3$ and $y=2$.
\end{itemize}\\[2pt]
\hline
\multicolumn{1}{|c|}{\quad \begin{tikzpicture}[line cap=round,line join=round,>=triangle 45,x=0.25cm,y=0.25cm]
\draw[->,color=black] (-0.5,0) -- (10.6,0);
\foreach \x in {,1,2,3,4,5,6,7,8,9,10}
\draw[shift={(\x,0)},color=black] (0pt,2pt) -- (0pt,-2pt);
\draw[color=black] (10.3,0.08) node [anchor=south west] { x};
\draw[->,color=black] (0,-0.46) -- (0,10.48);
\foreach \y in {,1,2,3,4,5,6,7,8,9,10}
\draw[shift={(0,\y)},color=black] (2pt,0pt) -- (-2pt,0pt);
\draw[color=black] (0.1,10.06) node [anchor=west] { y};
\clip(-0.5,-0.46) rectangle (10.6,10.48);
\fill [color=qqqqff] (1,1) circle (1.5pt);
\fill [color=qqqqff] (2,2) circle (1.5pt);
\fill [color=qqqqff] (3,3) circle (1.5pt);
\fill [color=qqqqff] (2,4) circle (1.5pt);
\fill [color=qqqqff] (2,6) circle (1.5pt);
\fill [color=qqqqff] (2,8) circle (1.5pt);
\fill [color=qqqqff] (2,10) circle (1.5pt);
\fill [color=qqqqff] (4,2) circle (1.5pt);
\fill [color=qqqqff] (6,2) circle (1.5pt);
\fill [color=qqqqff] (8,2) circle (1.5pt);
\fill [color=qqqqff] (10,2) circle (1.5pt);
\fill [color=qqqqff] (4,4) circle (1.5pt);
\fill [color=qqqqff] (3,5) circle (1.5pt);
\fill [color=qqqqff] (3,6) circle (1.5pt);
\fill [color=qqqqff] (3,7) circle (1.5pt);
\fill [color=qqqqff] (3,8) circle (1.5pt);
\fill [color=qqqqff] (3,9) circle (1.5pt);
\fill [color=qqqqff] (3,10) circle (1.5pt);
\fill [color=qqqqff] (5,3) circle (1.5pt);
\fill [color=qqqqff] (6,3) circle (1.5pt);
\fill [color=qqqqff] (7,3) circle (1.5pt);
\fill [color=qqqqff] (8,3) circle (1.5pt);
\fill [color=qqqqff] (9,3) circle (1.5pt);
\fill [color=qqqqff] (10,3) circle (1.5pt);
\fill [color=qqqqff] (5,4) circle (1.5pt);
\fill [color=qqqqff] (6,4) circle (1.5pt);
\fill [color=qqqqff] (7,4) circle (1.5pt);
\fill [color=qqqqff] (8,4) circle (1.5pt);
\fill [color=qqqqff] (9,4) circle (1.5pt);
\fill [color=qqqqff] (10,4) circle (1.5pt);
\fill [color=qqqqff] (4,5) circle (1.5pt);
\fill [color=qqqqff] (5,5) circle (1.5pt);
\fill [color=qqqqff] (6,5) circle (1.5pt);
\fill [color=qqqqff] (7,5) circle (1.5pt);
\fill [color=qqqqff] (8,5) circle (1.5pt);
\fill [color=qqqqff] (9,5) circle (1.5pt);
\fill [color=qqqqff] (10,5) circle (1.5pt);
\fill [color=qqqqff] (4,6) circle (1.5pt);
\fill [color=qqqqff] (5,6) circle (1.5pt);
\fill [color=qqqqff] (6,6) circle (1.5pt);
\fill [color=qqqqff] (7,6) circle (1.5pt);
\fill [color=qqqqff] (8,6) circle (1.5pt);
\fill [color=qqqqff] (9,6) circle (1.5pt);
\fill [color=qqqqff] (10,6) circle (1.5pt);
\fill [color=qqqqff] (4,7) circle (1.5pt);
\fill [color=qqqqff] (5,7) circle (1.5pt);
\fill [color=qqqqff] (6,7) circle (1.5pt);
\fill [color=qqqqff] (7,7) circle (1.5pt);
\fill [color=qqqqff] (8,7) circle (1.5pt);
\fill [color=qqqqff] (9,7) circle (1.5pt);
\fill [color=qqqqff] (10,7) circle (1.5pt);
\fill [color=qqqqff] (4,8) circle (1.5pt);
\fill [color=qqqqff] (5,8) circle (1.5pt);
\fill [color=qqqqff] (6,8) circle (1.5pt);
\fill [color=qqqqff] (7,8) circle (1.5pt);
\fill [color=qqqqff] (8,8) circle (1.5pt);
\fill [color=qqqqff] (9,8) circle (1.5pt);
\fill [color=qqqqff] (10,8) circle (1.5pt);
\fill [color=qqqqff] (4,9) circle (1.5pt);
\fill [color=qqqqff] (5,9) circle (1.5pt);
\fill [color=qqqqff] (6,9) circle (1.5pt);
\fill [color=qqqqff] (7,9) circle (1.5pt);
\fill [color=qqqqff] (8,9) circle (1.5pt);
\fill [color=qqqqff] (9,9) circle (1.5pt);
\fill [color=qqqqff] (10,9) circle (1.5pt);
\fill [color=qqqqff] (4,10) circle (1.5pt);
\fill [color=qqqqff] (5,10) circle (1.5pt);
\fill [color=qqqqff] (6,10) circle (1.5pt);
\fill [color=qqqqff] (7,10) circle (1.5pt);
\fill [color=qqqqff] (8,10) circle (1.5pt);
\fill [color=qqqqff] (9,10) circle (1.5pt);
\fill [color=qqqqff] (10,10) circle (1.5pt);
\end{tikzpicture} \quad} & \multicolumn{1}{|c|}{\quad \begin{tikzpicture}[line cap=round,line join=round,>=triangle 45,x=0.25cm,y=0.25cm]
\draw[->,color=black] (-0.5,0) -- (10.6,0);
\foreach \x in {,1,2,3,4,5,6,7,8,9,10}
\draw[shift={(\x,0)},color=black] (0pt,2pt) -- (0pt,-2pt);
\draw[color=black] (10.3,0.08) node [anchor=south west] { x};
\draw[->,color=black] (0,-0.46) -- (0,10.48);
\foreach \y in {,1,2,3,4,5,6,7,8,9,10}
\draw[shift={(0,\y)},color=black] (2pt,0pt) -- (-2pt,0pt);
\draw[color=black] (0.1,10.06) node [anchor=west] { y};
\clip(-0.5,-0.46) rectangle (10.6,10.48);
\fill [color=qqqqff] (1,1) circle (1.5pt);
\fill [color=qqqqff] (2,1) circle (1.5pt);
\fill [color=qqqqff] (3,3) circle (1.5pt);
\fill [color=qqqqff] (2,3) circle (1.5pt);
\fill [color=qqqqff] (2,5) circle (1.5pt);
\fill [color=qqqqff] (2,7) circle (1.5pt);
\fill [color=qqqqff] (2,9) circle (1.5pt);
\fill [color=qqqqff] (3,2) circle (1.5pt);
\fill [color=qqqqff] (4,2) circle (1.5pt);
\fill [color=qqqqff] (7,2) circle (1.5pt);
\fill [color=qqqqff] (8.2,1.92) circle (1.5pt);
\fill [color=qqqqff] (4,4) circle (1.5pt);
\fill [color=qqqqff] (3,5) circle (1.5pt);
\fill [color=qqqqff] (3,6) circle (1.5pt);
\fill [color=qqqqff] (3,7) circle (1.5pt);
\fill [color=qqqqff] (3,8) circle (1.5pt);
\fill [color=qqqqff] (3,9) circle (1.5pt);
\fill [color=qqqqff] (3,10) circle (1.5pt);
\fill [color=qqqqff] (5,3) circle (1.5pt);
\fill [color=qqqqff] (6,3) circle (1.5pt);
\fill [color=qqqqff] (7,3) circle (1.5pt);
\fill [color=qqqqff] (8,3) circle (1.5pt);
\fill [color=qqqqff] (9,3) circle (1.5pt);
\fill [color=qqqqff] (10,3) circle (1.5pt);
\fill [color=qqqqff] (5,4) circle (1.5pt);
\fill [color=qqqqff] (6,4) circle (1.5pt);
\fill [color=qqqqff] (7,4) circle (1.5pt);
\fill [color=qqqqff] (8,4) circle (1.5pt);
\fill [color=qqqqff] (9,4) circle (1.5pt);
\fill [color=qqqqff] (10,4) circle (1.5pt);
\fill [color=qqqqff] (4,5) circle (1.5pt);
\fill [color=qqqqff] (5,5) circle (1.5pt);
\fill [color=qqqqff] (6,5) circle (1.5pt);
\fill [color=qqqqff] (7,5) circle (1.5pt);
\fill [color=qqqqff] (8,5) circle (1.5pt);
\fill [color=qqqqff] (9,5) circle (1.5pt);
\fill [color=qqqqff] (10,5) circle (1.5pt);
\fill [color=qqqqff] (4,6) circle (1.5pt);
\fill [color=qqqqff] (5,6) circle (1.5pt);
\fill [color=qqqqff] (6,6) circle (1.5pt);
\fill [color=qqqqff] (7,6) circle (1.5pt);
\fill [color=qqqqff] (8,6) circle (1.5pt);
\fill [color=qqqqff] (9,6) circle (1.5pt);
\fill [color=qqqqff] (10,6) circle (1.5pt);
\fill [color=qqqqff] (4,7) circle (1.5pt);
\fill [color=qqqqff] (5,7) circle (1.5pt);
\fill [color=qqqqff] (6,7) circle (1.5pt);
\fill [color=qqqqff] (7,7) circle (1.5pt);
\fill [color=qqqqff] (8,7) circle (1.5pt);
\fill [color=qqqqff] (9,7) circle (1.5pt);
\fill [color=qqqqff] (10,7) circle (1.5pt);
\fill [color=qqqqff] (4,8) circle (1.5pt);
\fill [color=qqqqff] (5,8) circle (1.5pt);
\fill [color=qqqqff] (6,8) circle (1.5pt);
\fill [color=qqqqff] (7,8) circle (1.5pt);
\fill [color=qqqqff] (8,8) circle (1.5pt);
\fill [color=qqqqff] (9,8) circle (1.5pt);
\fill [color=qqqqff] (10,8) circle (1.5pt);
\fill [color=qqqqff] (4,9) circle (1.5pt);
\fill [color=qqqqff] (5,9) circle (1.5pt);
\fill [color=qqqqff] (6,9) circle (1.5pt);
\fill [color=qqqqff] (7,9) circle (1.5pt);
\fill [color=qqqqff] (8,9) circle (1.5pt);
\fill [color=qqqqff] (9,9) circle (1.5pt);
\fill [color=qqqqff] (10,9) circle (1.5pt);
\fill [color=qqqqff] (4,10) circle (1.5pt);
\fill [color=qqqqff] (5,10) circle (1.5pt);
\fill [color=qqqqff] (6,10) circle (1.5pt);
\fill [color=qqqqff] (7,10) circle (1.5pt);
\fill [color=qqqqff] (8,10) circle (1.5pt);
\fill [color=qqqqff] (9,10) circle (1.5pt);
\fill [color=qqqqff] (10,10) circle (1.5pt);
\fill [color=qqqqff] (3,4) circle (1.5pt);
\end{tikzpicture} \quad} & \multicolumn{1}{|c|}{\quad \begin{tikzpicture}[line cap=round,line join=round,>=triangle 45,x=0.25cm,y=0.25cm]
\draw[->,color=black] (-0.5,0) -- (10.6,0);
\foreach \x in {,1,2,3,4,5,6,7,8,9,10}
\draw[shift={(\x,0)},color=black] (0pt,2pt) -- (0pt,-2pt);
\draw[color=black] (10.3,0.08) node [anchor=south west] { x};
\draw[->,color=black] (0,-0.46) -- (0,10.48);
\foreach \y in {,1,2,3,4,5,6,7,8,9,10}
\draw[shift={(0,\y)},color=black] (2pt,0pt) -- (-2pt,0pt);
\draw[color=black] (0.1,10.06) node [anchor=west] { y};
\clip(-0.5,-0.46) rectangle (10.6,10.48);
\fill [color=qqqqff] (1,1) circle (1.5pt);
\fill [color=qqqqff] (2,1) circle (1.5pt);
\fill [color=qqqqff] (3,3) circle (1.5pt);
\fill [color=qqqqff] (2,6) circle (1.5pt);
\fill [color=qqqqff] (2,5) circle (1.5pt);
\fill [color=qqqqff] (2,10) circle (1.5pt);
\fill [color=qqqqff] (2,9) circle (1.5pt);
\fill [color=qqqqff] (1,2) circle (1.5pt);
\fill [color=qqqqff] (2,2) circle (1.5pt);
\fill [color=qqqqff] (5,2) circle (1.5pt);
\fill [color=qqqqff] (6,2) circle (1.5pt);
\fill [color=qqqqff] (4,4) circle (1.5pt);
\fill [color=qqqqff] (3,5) circle (1.5pt);
\fill [color=qqqqff] (3,6) circle (1.5pt);
\fill [color=qqqqff] (3,7) circle (1.5pt);
\fill [color=qqqqff] (3,8) circle (1.5pt);
\fill [color=qqqqff] (3,9) circle (1.5pt);
\fill [color=qqqqff] (3,10) circle (1.5pt);
\fill [color=qqqqff] (5,3) circle (1.5pt);
\fill [color=qqqqff] (6,3) circle (1.5pt);
\fill [color=qqqqff] (7,3) circle (1.5pt);
\fill [color=qqqqff] (8,3) circle (1.5pt);
\fill [color=qqqqff] (9,3) circle (1.5pt);
\fill [color=qqqqff] (10,3) circle (1.5pt);
\fill [color=qqqqff] (5,4) circle (1.5pt);
\fill [color=qqqqff] (6,4) circle (1.5pt);
\fill [color=qqqqff] (7,4) circle (1.5pt);
\fill [color=qqqqff] (8,4) circle (1.5pt);
\fill [color=qqqqff] (9,4) circle (1.5pt);
\fill [color=qqqqff] (10,4) circle (1.5pt);
\fill [color=qqqqff] (4,5) circle (1.5pt);
\fill [color=qqqqff] (5,5) circle (1.5pt);
\fill [color=qqqqff] (6,5) circle (1.5pt);
\fill [color=qqqqff] (7,5) circle (1.5pt);
\fill [color=qqqqff] (8,5) circle (1.5pt);
\fill [color=qqqqff] (9,5) circle (1.5pt);
\fill [color=qqqqff] (10,5) circle (1.5pt);
\fill [color=qqqqff] (4,6) circle (1.5pt);
\fill [color=qqqqff] (5,6) circle (1.5pt);
\fill [color=qqqqff] (6,6) circle (1.5pt);
\fill [color=qqqqff] (7,6) circle (1.5pt);
\fill [color=qqqqff] (8,6) circle (1.5pt);
\fill [color=qqqqff] (9,6) circle (1.5pt);
\fill [color=qqqqff] (10,6) circle (1.5pt);
\fill [color=qqqqff] (4,7) circle (1.5pt);
\fill [color=qqqqff] (5,7) circle (1.5pt);
\fill [color=qqqqff] (6,7) circle (1.5pt);
\fill [color=qqqqff] (7,7) circle (1.5pt);
\fill [color=qqqqff] (8,7) circle (1.5pt);
\fill [color=qqqqff] (9,7) circle (1.5pt);
\fill [color=qqqqff] (10,7) circle (1.5pt);
\fill [color=qqqqff] (4,8) circle (1.5pt);
\fill [color=qqqqff] (5,8) circle (1.5pt);
\fill [color=qqqqff] (6,8) circle (1.5pt);
\fill [color=qqqqff] (7,8) circle (1.5pt);
\fill [color=qqqqff] (8,8) circle (1.5pt);
\fill [color=qqqqff] (9,8) circle (1.5pt);
\fill [color=qqqqff] (10,8) circle (1.5pt);
\fill [color=qqqqff] (4,9) circle (1.5pt);
\fill [color=qqqqff] (5,9) circle (1.5pt);
\fill [color=qqqqff] (6,9) circle (1.5pt);
\fill [color=qqqqff] (7,9) circle (1.5pt);
\fill [color=qqqqff] (8,9) circle (1.5pt);
\fill [color=qqqqff] (9,9) circle (1.5pt);
\fill [color=qqqqff] (10,9) circle (1.5pt);
\fill [color=qqqqff] (4,10) circle (1.5pt);
\fill [color=qqqqff] (5,10) circle (1.5pt);
\fill [color=qqqqff] (6,10) circle (1.5pt);
\fill [color=qqqqff] (7,10) circle (1.5pt);
\fill [color=qqqqff] (8,10) circle (1.5pt);
\fill [color=qqqqff] (9,10) circle (1.5pt);
\fill [color=qqqqff] (10,10) circle (1.5pt);
\fill [color=qqqqff] (3,4) circle (1.5pt);
\fill [color=qqqqff] (9,2) circle (1.5pt);
\fill [color=qqqqff] (10,2) circle (1.5pt);
\fill [color=qqqqff] (4,3) circle (1.5pt);
\end{tikzpicture} \quad} \\[5pt]
\hline
\multicolumn{3}{|p{12cm}|}{For $i$ even, $j$ odd the set $FCS (i,j)$ is obtained from $FCS (j,i)$ by the reflection with respect to the line $x=y$.}\\[2pt]
\hline
\end{tabular}
\end{center}
\end{table}

\begin{corollary} \label{cor1}
Assume that $1<p<m-2$. Then the group
$E^m_{\mathrm{B}}(\sqcup_{k=1}^2 S^{p})$
is infinite if and only if $(m-3)/(m-p-2)$ is an integer, which is distinct from $5$ for $m-p$ odd, and distinct from $3,5,7$ for $m-p$ even.
%
\end{corollary}

\begin{example} \label{rem-E558}
Applying Corollary~\ref{cor1} and Theorem~\ref{thm-1a} we see that the group $E_{\rm U}^8(\sqcup_{k=1}^r S^5)$ is infinite if and only if $r \geq 3$.  This corrects an error 
in~\cite[Corollary~3.18]{Schmitt-02} where it is stated that the group $L_b := E_{\rm U}^8(\sqcup_{k=1}^b S^5)$ is infinite if and only if $b \geq 2$; 
see Remark \ref{rem:Hil-Mil} for an explanation.  Note that~\cite[Corollary~3.19]{Schmitt-02} should also be modified by changing $b \geq 2$ to $b \geq 3$.
\end{example}

Combining Theorems~\ref{th1}--\ref{thm-1b}
we obtain the following definitive finiteness criteria for $E^m(\sqcup_{k=1}^r S^{p_k})$.

\begin{corollary} \label{cor-globalfinite}
Assume that $p_1,\dots,p_r<m-2$. Then the group $E^m(\sqcup_{k=1}^r S^{p_k})$ is infinite if and only if there is a subsequence $(k_1,\dots,k_s)\subset (1,\dots,r)$ satisfying one of the following conditions:
\begin{itemize}
\item
$s=1$, $4\,|\,p_{k_1}+1$, and $m<3p_{k_1}/2 + 2$;
\item
$s=2$ and there is $(x_1, x_2)\in FCS (m-p_{k_1}, m - p_{k_2})$ such that $(m- p_{k_1} - 2) x_1 + (m - p_{k_2} - 2)x_2 = m-3$;
\item
$s\ge3$ and the equation $(m - p_{k_1} - 2)x_1 + \dots + (m - p_{k_s} - 2)x_s = m-3$ has a solution in positive integers.
\end{itemize}
\end{corollary}


%
%
%




\subsection{Formulae for the rank of the group of links} \label{ssec-ranks}



In this subsection we state results on ranks of the groups of links. We give an explicit formula for the ranks of the groups $E^m_\mathrm{B}(\sqcup_{k=1}^r S^{p_k})$ and $E^m_{}(\sqcup_{k=1}^r S^{p_k})$ in Theorem~\ref{thm-general} below.
The formula for Brunnian links asserts that the rank is equal to the number of solutions
of the equation from Theorem~\ref{thm-1a}
counted with certain ``multiplicities'' $m(x_1,\dots,x_r)$.
As a corollary to Theorem \ref{thm-general}, we obtain a formula for the rank of the group of links with components having the same dimension (Theorem~\ref{th2}).  Some computations are shown in Table~\ref{tab-ranks} below.  There is a computer application available based on Theorem \ref{thm-general} which computes these ranks in general~\cite{Sko11D}.


To state our results we need the following notation.
Denote the rank of a finitely generated abelian group $G$
by $\rank G$ and let $\mathbb{N}$ denote the set of positive integers. 
The \emph{M\"obius function} is defined by the formula:
\[
\mu(i) = \left\{
\begin{array}{cl}
1, & \text{if } i = 1; \\
(-1)^k, & \text{if } i=q_1\, q_2 \dots q_k \text{~is a product of distinct primes;} \\
0, & \text{if $i$ is not square free.}
\end{array}
\right.
\]
Denote by $\gcd(x_1,\dots,x_r)$ the greatest common divisor of integers $x_1$, \dots, $x_r$.
Denote:
\begin{align}
g(x_1,\dots,x_r)&=(m-p_1-2)x_1+\dots+(m-p_r-2)x_r,\label{eq2}\\
d(x_1,\dots,x_r)&=\frac{(-1)^{g(x_1,\dots,x_r)}}{x_1+\dots+x_r}\sum_{i\,|\,\gcd(x_1,\dots,x_k)} \mu(i)(-1)^{{g(x_1,\dots,x_r)}/{i}}\cdot\frac{\left((x_1+\dots+x_r)/i\right)!}{(x_1/i)!\cdots (x_r/i)!}, \label{eq3}\\
m(x_1,\dots,x_r)&=d(x_1-1,x_2,\dots,x_r)+
\dots+d(x_1,x_2,\dots,x_r-1)-d(x_1,x_2\dots,x_r)\label{eq4}.
\end{align}
Set $d(0,\dots,0)=1$ and $d(x_1,\dots,x_r)=0$, if at least one of the numbers $x_1,\dots,x_k$ is negative.
For each positive integer $t$ define
the following polynomials in the indeterminate $r$:
\begin{equation}\label{eq5}
w_{t}(r) = \frac{1}{t}\sum_{i|t}\mu(i)r^{t/i}
\qquad\text{and}\qquad
w_{t,s}(r) =
\begin{cases}
w_t(r)+w_{t/2}(r), &\text{if $s$ is odd and $t=2\!\!\mod4$;}\\
w_t(r), &\text{otherwise.}
\end{cases}
\end{equation}
Set $w_{t,s}(r)=0$, if $t\not\in\mathbb{N}$.
Set
$$
c_{p,m}=\begin{cases}
1, &\text{if }4\,|\, p+1\text{ and }m<3p/2+2;\\
0, &\text{otherwise.}
\end{cases}
$$
Set $\delta_{i,j}=1$, if $i=j$, and $\delta_{i,j}=0$, otherwise.


\begin{theorem}\label{thm-general}
Assume that $r>1$ and $p_1,\dots,p_r< m-2$.
Then 
\begin{equation}\label{eq1}
\rank E^m_{\mathrm{B}}(\sqcup_{k=1}^r S^{p_k})=
\sum_{
x_1,\dots,x_r\in \mathbb{N}\,:\, g(x_1,\dots,x_r)=m-3
}
m(x_1,\dots,x_r).
\end{equation}
\begin{equation}\label{eq1b}
\rank E^m_{\mathrm{}}(\sqcup_{k=1}^r S^{p_k})= \sum_{
x_1,\dots,x_r\in \mathbb{N}\cup\{0\}\,:\, g(x_1,\dots,x_r)=m-3
}
m(x_1,\dots,x_r)+\sum_{k=1}^r
\left(c_{p_k,m}-\delta_{2(m-3)/(m-p_k-2),5-(-1)^{m-p_k}}\right)\!.
\end{equation}
\end{theorem}

It is not obvious when the sum~(\ref{eq1}) equals zero; see Section~\ref{ssec-fin-proofs}.

\begin{theorem}\label{th2}
Assume that $1<p<m-2$. Denote $t=(m-3)/(m-p-2)$.
Then
\begin{equation*}
\mathrm{rk} E^m(\sqcup_{k=1}^r S^p)=
r\cdot \left(w_{t-1,m-p}(r)+c_{p,m}
-\delta_{2t,5-(-1)^{m-p}}\right)-w_{t,m-p}(r).
\end{equation*}
\end{theorem}




\begin{corollary}\label{cor-easy}
The rank of the group $E_{\mathrm{B}}^m(S^{m-3}\sqcup S^{m-3})$ tends to infinity as $m$ tends to infinity.
\end{corollary}





\begin{table}[ht]
\caption{The ranks of the groups $E_{\mathrm{B}}^{p+l+k}(S^p\sqcup S^{p+k})$ for $p\le 5$.
}
\label{tab-ranks}
\begin{center}
\begin{tabular}{|l|c|c|cc|ccc|cccc|ccccc|}
\hline
\multicolumn{2}{|l|}{$p$} & 1 &
\multicolumn{2}{|c|}{2} &
\multicolumn{3}{|c|}{3} &
\multicolumn{4}{|c|}{4} &
\multicolumn{5}{|c|}{5}\\
\hline
\multicolumn{2}{|l|}{$l$} & $\ge 3$ &
3 & \multicolumn{1}{|l|}{$\ge 4$} &
3 & 4 & \multicolumn{1}{|l|}{$\ge 5$} &
3 & 4 & 5 & \multicolumn{1}{|l|}{$\ge 6$} &
3 & 4 & 5 & 6 & \multicolumn{1}{|l|}{$\ge 7$}\\
\hline
    & 0 &
0 & 1 & 0 & 2 & 1 & 0 & 1 & 0 & 1 & 0 & 0 & 0 & 0 & 1 & 0 \\
$k$ & 1 &
0 & 1 & 0 & 1 & 1 & 0 & 0 & 0 & 1 & 0 & 0 & 0 & 0 & 1 & 0 \\
    & 2 &
0 & 1 & 0 & 1 & 1 & 0 & 0 & 0 & 1 & 0 & 1 & 0 & 0 & 1 & 0 \\
\hhline{|~-|~|~~|~~~|~~~~|~~~~~|}
    & $\ge3$ &
0 & 1 & 0 & 1 & 1 & 0 & 0 & 0 & 1 & 0 & 0 & 0 & 0 & 1 & 0 \\
\hline
\end{tabular}
%
\end{center}
\end{table}


\subsection{Framed links}\label{ssec-framed}

Let $l_1,\dots,l_r$ be integers such that $0\le l_k\le m-p_k$ for each $k=1,\dots,r$.
Denote by
\[
E^m(\sqcup_{k=1}^r S^{p_k} \times D^{l_k}):=
\{\, f \colon \sqcup_{k=1}^r S^{p_k} \times D^{l_k} \hookrightarrow S^m\,\}/\text{isotopy}
\]
the set of smooth isotopy classes of smooth \emph{orientation preserving} embeddings $\sqcup_{k=1}^r S^{p_k} \times D^{l_k} \hookrightarrow S^m$ (for $l_k\ne m-p_k$ any
embedding $S^{p_k} \times D^{l_k}\hookrightarrow S^m$ is considered to be orientation preserving). For $p_1,\dots,p_r<m-2$ the set $E^m(\sqcup_{k=1}^r S^{p_k} \times D^{l_k})$ is a group
called the group of \emph{partially} \emph{framed links}. This notion generalizes both links ($l_k=0$) and \emph{framed links} ($l_k=m-p_k$) \cite{Hae66A}. Partially framed links play an important role in the classification of embeddings of general manifolds \cite{CRS08,Sko08,Sko11}.

To give a formula for the rank of the group of partially framed links we first recall that the \emph{Stiefel manifold} $V_{q,l}$ is the manifold of all $l$-tuples of pairwise orthogonal unit vectors in $\mathbb{R}^q$; by definition $V_{q, 0}$ is a point.  The ranks of the homotopy groups $V_{q, l}$ are essentially known; see Section~\ref{ssec-framed-proofs}:

\begin{lemma}\label{lstiefel} Assume $1\le l\le q$. Then
$$
\rank \pi_p(V_{q,l})=
\begin{cases}
2, &\text{if }4\,|\,p+1=q\le 2l;\\
1, &\text{if }4\,|\,p+1\ne q\text{ and }p/2+1<q<l+p/2+1   ,\\
   &\text{or }4\,|\,p+1=q>2l,\\
   &\text{or }4\,|\,p-1=q-2,\\
   &\text{or }2\,|\,p=q-l;\\
0, &\text{otherwise.}
\end{cases}
$$
\end{lemma}

\begin{theorem} \label{th-framed}
Assume $p_k<m-2$ and $0\le l_k\le m-p_k$. Then there is an equality
\[ \mathrm{rk} E^m(\sqcup_{k=1}^r S^{p_k} \times D^{l_k})= \mathrm{rk} E^m(\sqcup_{k=1}^r S^{p_k})+\sum_{k=1}^r \mathrm{rk}\pi_{p_k}(V_{m-p_k,l_k}).\]
\end{theorem}

Combining statements~\ref{th1}, \ref{th-framed}, \ref{lstiefel}, and~\ref{cor-globalfinite} we obtain the following corollaries.

\begin{corollary}\label{cor-framedfinite}
Assume $p<m-2$ and $1\le l\le m-p$. Then the set $E^m(S^p\times D^l)$ is infinite if and only if at least one of the following conditions holds:
\begin{itemize}
\item $4\,|\,p+1$ and $m<3p/2+l+1$;
\item $2\,|\,p+1$ and $m=2p+1$;
\item $2\,|\,p$ and $m=2p+l$.
\end{itemize}
\end{corollary}

\begin{corollary}\label{cor-globalframed}
Assume that $p_1,\dots,p_r<m-2$. Then the set $E^m(\sqcup_{k=1}^r S^{p_k}\times D^{m-p_k})$ is infinite if and only if there is a subsequence $(k_1,\dots,k_s)\subset (1,\dots,r)$ satisfying one of the following conditions:
\begin{itemize}
\item
$s=1$ and either $4\,|\,p_{k_1}+1$ or $4\,|\,m+1=2p_{k_1}+2$;
\item
$s=2$ and there is $(x_1, x_2)\in FCS (m-p_{k_1}, m - p_{k_2})$ such that $(m- p_{k_1} - 2) x_1 + (m - p_{k_2} - 2)x_2 = m-3$;
\item
$s\ge3$ and the equation $(m - p_{k_1} - 2)x_1 + \dots + (m - p_{k_s} - 2)x_s = m-3$ has a solution in positive integers.
\end{itemize}
\end{corollary}

\subsection{Other types of links}

For $p_1,\dots,p_r<m-2$ the group of piecewise-linear (respectively, continuous) isotopy classes of piecewise-linear (respectively, continuous) embeddings $\sqcup_{k=1}^{r} S^{p_k}\to S^m$ is isomorphic to the group $E^m_{\mathrm{U}}(\sqcup_{k=1}^{r} S^{p_k})$ by \cite[Section~2.6]{Hae66C} (respectively, \cite{Br72}).
Thus our results provide the rational classification of piecewise-linear and topological links as well.

\subsection{The organization of the paper and its relationship to other work}


We give the proofs of all the results stated above in Section \ref{sec-proofs}:  We start by recalling the Haefliger sequence in Section \ref{ssec-Hae}. Using this sequence at the beginning of Section~\ref{ssec-ranks-proofs} we convert the problem of the rational classification of links into a completely algebraic problem in the theory of Lie algebras.  In the remainder of Section~\ref{ssec-ranks-proofs} we solve this problem and so obtain the formulae for ranks of the groups of links.  In this way we convert the question of which groups of links are finite into a problem of elementary number theory.  We present a solution to this latter problem in Section~\ref{ssec-fin-proofs}. Using our results for unframed links and some standard algebraic topology we obtain the rational classification of framed links in Section~\ref{ssec-framed-proofs}.


In Sections~\ref{ssec-handle}, \ref{ssec-mapping}, and further in~\cite{CFS11} we apply our calculations to the classification of handlebodies, thickenings and to the computation of mapping class groups. In~\cite{Sko11} the third author applies our calculations to the rational classification of embeddings of a product of two spheres;
see \cite{CRS07, CRS08, Sko08} for particular cases.
In Section~\ref{ssec-open} we discuss some open problems.

\section{Proofs} \label{sec-proofs}



In this section we prove the results stated in the introduction. The results stated in Sections~\ref{ssec-fin}, \ref{ssec-ranks},  and ~\ref{ssec-framed} are proven in Sections~\ref{ssec-fin-proofs}, \ref{ssec-ranks-proofs}, and~\ref{ssec-framed-proofs}, respectively.

\subsection{The Haefliger sequence for Brunnian links} \label{ssec-Hae}
By Theorem~\ref{th-structure} the group of links splits as the sum of the groups of Brunnian links.  There is a Haefliger sequence for Brunnian links, \eqref{seq:PHae} below, which is analogous to the
Haefliger sequence~\eqref{seq:UHae}.  In this subsection we recall the groups in the Haefliger sequence for Brunnian links and also the key homomorphism $w$: we refer the reader to \cite[\S9.4]{Hae66C} and \cite[\S1.2]{Ne82} for further details.  We also give the definition of the groups in the Haefliger sequence~\eqref{seq:UHae} and state the Hilton--Milnor theorem which plays a key role in all computations using these sequences.



Denote by $\pi_g(X)$ the $g$-th homotopy group of a space $X$. For $i=0,\dots,r$ there are obvious retractions $\hat R_i \colon \! \! \vee_{k=1}^r S^{m-p_k-1} \to \vee _{1\le k\le r,\,k\ne i} S^{m-p_k-1}$ obtained by collapsing the $i$-th sphere $S^{m-p_i-1}$ from the wedge.  Taking the kernel of the sum over $i$ of the homomorphism induced on $\pi_{g+1}$ one obtains the finitely generated abelian group
\[
{}^i \Pi^{(q)}_{g+1} : = \kernel \left\{ \pi_{g+1}\left(\bigvee_{1\le k\le r} S^{m-p_k-1}\right)\to
\bigoplus_{1\le j\le r,\,j\ne i} \pi_{g+1}\left(\bigvee_{1\le k\le r,\,k\ne j} S^{m-p_k-1}\right) \right\}.
\]
In addition, define
\[ {}^0\! \Lambda^{(q)}_{(p)}:= \bigoplus_{1\le k\le r}{}^k \Pi^{(q)}_{p_k}\]
and define the homomorphism
\[  w\colon {}^0\! \Lambda^{(q)}_{(p)}\to {}^0 \Pi^{(q)}_{g+1}, \quad (u_1,\dots,u_r) \mapsto [u_1,\i_1]+\dots+[u_r,\i_r],\]
where $\i_i\colon S^{m-p_i-1}\to \vee_{k=1}^r S^{m-p_k-1}$ is the obvious inclusion and $[\cdot,\cdot]$ is the Whitehead product.
The Haefliger sequence for Brunnian links is a long exact sequence of finitely generated abelian groups which runs as follows:
\begin{equation} \label{seq:PHae}
\dots \to {}^0\Pi_{m-1}^{(q)} \dlra{\mu} E^m_{\mathrm{B}}(\sqcup_{k=1}^r S^{p_k}) \dlra{\lambda}  {}^0 \! \Lambda_{(p)}^{(q)} \dlra{w} {}^0\Pi_{m-2}^{(q)} \dlra{\mu}  E^{m-1}_{\mathrm{B}}(\sqcup_{k=1}^r S^{p_k-1}) \to \dots
\end{equation}
Note that we use the same notation for the maps in both sequences~\eqref{seq:PHae} and~\eqref{seq:UHae}: no confusion will arise from this.

The groups in the Haefliger sequence~\eqref{seq:UHae} are defined by similar formulae:
\begin{align*}
\Pi_{g}^{(q)}&=\kernel\left(\pi_{g}(\vee_{k=1}^r S^{m-p_k-1})\to\oplus_{k=1}^r \pi_{g}(S^{m-p_k-1})\right),\\
\Lambda_{(p)}^{(q)}&=\oplus_{i=1}^r\kernel\left(\pi_{p_i}(\vee_{k=1}^r S^{m-p_k-1})\to \pi_{p_i}(S^{m-p_i-1})\right).\\
\end{align*}
\vskip-0.8cm
The homomorphism $w \colon \Lambda_{(p)}^{(q)} \to \Pi_{g}^{(q)}$ is also defined
analogously via Whitehead products.
%

The homotopy groups in the sequences~\eqref{seq:UHae} and~\eqref{seq:PHae} can be expressed in terms of the homotopy groups of spheres using the \emph{Hilton--Milnor theorem} which we now recall; see \cite[\S2.1]{Ne82} or \cite[Section 5 of Chapter IV]{Serre-69} for the necessary definitions.  Let $B$ be the Hall basis in a free graded Lie algebra (not to be confused with a superalgebra) generated by elements $P_1,\dots,P_r$ of degrees $m-p_1-2,\dots,m-p_r-2$. Denote by $g(b)$ the degree of an element $b\in B$. The Hilton-Milnor theorem states that there is an isomorphism
$$
\pi_g(\vee_{k=1}^r S^{m-p_k-1})\cong \bigoplus_{b\in B} \pi_g(S^{g(b)+1}).
$$
The isomorphism from the the right to the left is given by the formula $\{x_b\}_{b\in B}\mapsto \sum \alpha(b)\circ x_b$,
where for a product $b$ of generators $P_1,\dots,P_r$ we denote by $\alpha(b)$ the analogous Whitehead product of the homotopy classes $\i_1,\dots,\i_r$, and $\circ$ denotes the composition of homotopy classes.

\begin{remark} \label{rem:Hil-Mil}
In general the map $b\mapsto\alpha(b)$ does \emph{not} extend to a multiplicative homomorphism.
For instance,  for $m-p_1$
odd we have $[\alpha(P_1),\alpha(P_1)]=[\i_1,\i_1]\ne 0=\alpha(0)=\alpha([P_1,P_1])$.  This explains the mistake in the proof \cite[Corollary~3.18]{Schmitt-02} where the map $b \mapsto \alpha(b)$ is assumed to be multiplicative.
%
 \end{remark}


\subsection{Computation of ranks} \label{ssec-ranks-proofs}\label{ssec-haefliger}


In this subsection we prove Theorems~\ref{thm-general} and~\ref{th2}.
Theorem~\ref{thm-general} follows from assertions~\ref{thhae}
and~\ref{lker} below. Theorem~\ref{th2} can be obtained from Theorem~\ref{thm-general} but we give a short alternative proof.

A {\it graded Lie superalgebra}\footnotemark\ over $\mathbb{Q}$ is a graded vector space $L=\bigoplus_{g=0}^{\infty}L_g$, along with a bilinear operation
$[\cdot,\cdot]\colon L\otimes L \to L$ satisfying the following axioms:
\begin{enumerate}
\item \emph{Respect of the grading}: $[L_i,L_j]\subseteq L_{i+j}$.
\item \emph{Symmetry}: if $u \in L_i$, $v \in L_j$ then $[u,v]-(-1)^{(i+1)(j+1)}\,[v,u]=0$.
\item \emph{Jacobi identity}: if $u\in L_i$, $v\in L_j$, $w\in L_k$ then
$$
( - 1)^{(i+1)(k+1)}[[u,v],w] + ( - 1)^{(j+1)(i+1)}[[v,w],u] + ( - 1)^{(k+1)(j+1)}[[w,u],v] = 0.
$$
\end{enumerate}

\footnotetext{
Some authors use another definition, which can be obtained from ours replacing $[x,y]$ by $(-1)^{\deg x}[x,y]$. Both definitions lead to
the same dimensions of homogeneous components of free graded Lie superalgebras.
}

A \emph{free graded Lie superalgebra} generated by a set of elements with given degrees 
is the quotient of the free algebra generated by the set by the ideal generated by the terms from the left parts of the axioms~(2)--(3). The grading of this quotient is uniquely defined by axiom~(1). Hereafter denote by $L=\bigoplus_{g=0}^{\infty}L_g$ the free Lie superalgebra generated by the elements $P_1,\dots,P_r$ of degrees $m-p_1-2,\dots,m-p_r-2$.

Denote by $L^{0}$ the subalgebra of $L$ generated by the products containing each generator $P_1,\dots,P_r$
at least once. Denote by $L^{i}$ the subalgebra generated by the products containing each generator at least once except possibly $P_i$. Denote by $L^i_g$ the subspace of $L^{i}$ formed by elements of degree $g$.
For a finitely generated abelian group $G$ identify $G\otimes \mathbb{Q}=\mathbb{Q}^{\rank G}$.

Using this notation we can state a version of the Haefliger sequence~\eqref{seq:PHae} in purely algebraic terms:

\begin{theorem}\label{thhae}
For $r>1$ and $p_1,\dots,p_r<m-2$ there is an exact sequence
$$
\dots                                             \to
\oplus_{k=1}^r L^{k}_{p_k}                        \xrightarrow{w}
L^{0}_{m-2}                                       \to
E^m_{\mathrm{B}}(\sqcup_{k=1}^r S^{p_k})\otimes \mathbb{Q}\to
\oplus_{k=1}^r L^{k}_{p_k-1}                      \xrightarrow{w}
L^{0}_{m-3}                                       \to
\dots
$$
The linear map $w$
in the sequence is given by the formula
$w(u_1,\dots,u_r)=[u_1,P_1]+\dots+[u_r,P_r]$.
\end{theorem}

\begin{proof}[Proof of Theorem~\ref{thhae}]
We need only rewrite the sequence \eqref{seq:PHae} tensored by $\Q$ in terms of Lie superalgebras.
For a simply connected finite CW-complex $X$
the group $\bigoplus_{g=0}^{\infty}\pi_{g+1}(X)\otimes \mathbb{Q}$
is a graded Lie superalgebra with respect to the Whitehead product operation:
\[ [\cdot,\cdot] \colon (\pi_i(X) \tensor \Q) \times (\pi_j(X) \tensor \Q) \to \pi_{i + j - 1}(X) \tensor \Q. \]
Note that the degree of a homotopy class is one less than its dimension.




A rational version 
of the Hilton--Milnor Theorem \textup{\cite[p.\,116]{GM81}} states that
the graded Lie superalgebra
\[ \bigoplus_{g=0}^{\infty}\pi_{g+1}(\vee_{k=1}^r S^{m-p_k-1})\otimes \mathbb{Q}\]
is isomorphic to the free Lie superalgebra $L$. The isomorphism takes the homotopy class of each inclusion $\i_i\colon S^{m-p_i-1}\to \vee_{k=1}^r S^{m-p_k-1}$ to the generator~$P_i$.
%
Thus ${}^i\Pi^{(q)}_{g+1}\otimes\mathbb{Q}\cong L^{i}_g $
for each $i=0,\dots,r$ and the theorem follows.
\end{proof}


Now the rational classification of links reduces to the following purely algebraic result.

\begin{lemma}\label{lker}
\noindent\textup{(a)} The map $w$ from Theorem~\ref{thhae} is surjective.

\noindent\textup{(b)} The dimension of the kernel of $w$ equals the expression on the right hand side of~\textup{(\ref{eq1})} .
%
\end{lemma}

\begin{proof}[Proof of Lemma~\ref{lker}(a)]
Take $t\in L^0_{m-3}$. Let us prove that $t$ belongs to the image of $w$. It suffices to consider the case when $t$ is a product of generators. Since $r>1$ and $t\in L^0$ it follows that $t$ itself is not a generator. Thus $t=[u,v]$ for some products $u$ and $v$ of generators. So assertion~(a) reduces to the following claim.
\end{proof}

\begin{claim}\label{cla5} Let $1\le i\le r$ and let $u$ be a product of generators $P_i,\dots,P_r$ (each of the generators may appear in the product several times). Then for any $v\in L$ there are $v_i,\dots,v_r\in L$ such that $[u,v]=[P_i,v_i]+\dots+[P_r,v_r]$.
\end{claim}

\begin{proof}[Proof of Claim~\ref{cla5}]
Let us prove the claim by induction over the number of factors in the product $u$.
If there is only one factor then $u$ is a generator itself, and there is nothing to prove. Otherwise $u=[u_1,u_2]$ for some products $u_1,u_2\in L$ containing fewer factors than $u$. By the symmetry and the Jacobi identity we get $[u,v]=\pm[u_1,[u_2,v]]\pm[u_2,[u_1,v]]$.
Apply the inductive hypothesis for $u':=u_1$ and $v':=[u_2,v]$. Since $u'$ contains fewer factors than $u$
it follows that $[u_1,[u_2,v]]=[u',v']=[P_i,v'_i]+\dots+[P_r,v'_r]$
for some $v'_i,\dots,v'_r\in L$.
Analogously, since $u_2$ contains fewer factors than $u$
it follows that $[u_2,[u_1,v]]=[P_i,v''_i]+\dots+[P_r,v''_r]$
for some $v''_i,\dots,v''_r\in L$.
Thus $[u,v]=\pm[u_1,[u_2,v]]\pm[u_2,[u_1,v]]=[P_i,\pm v'_i\pm v''_i]+\dots+[P_r,\pm v'_r\pm v''_r]$,
which proves the claim.
\end{proof}


\begin{proof}[Proof of Lemma~\ref{lem:Hae-splits}]
This follows directly from Lemma~\ref{lker}.a and Theorem~\ref{th-structure}.
\end{proof}

We say that an element $t\in L$ has {\it multidegree} $(x_1,\dots,x_r)$ if $t$ is a product of generators in which the generator $P_k$ appears exactly $x_k$ times
for each $k=1,\dots,r$. The multidegree does not depend on the choice of representation as the product because axioms~(2) and~(3) above are relations between elements of the same multidegree. The element $t$ of multidegree $(x_1,\dots,x_r)$ has degree $g(x_1,\dots,x_r)$ given by formula~(\ref{eq2}). Denote by $L_{x_1,\dots,x_r}$ the linear subspace of $L$ spanned by elements of multidegree $(x_1,\dots,x_r)$.

\begin{theorem}\label{thlie} \textup{(See \cite[Corollary~1.1(3)]{Pet03})}
The dimension $d(x_1,\dots,x_r)$ of the space
$L_{x_1,\dots,x_r}$ is given by the formula~\textup{(\ref{eq3})}.
\end{theorem}

\begin{proof}[Proof of Lemma~\ref{lker}(b)] By part (a) we have
$
\dim\kernel w = \dim L^{1}_{p_1-1} +\dots+ \dim L^{b}_{p_r-1} - \dim L^{0}_{m-3}
$.
The space $L^i_g$ is a direct sum of all the spaces $L_{x_1,\dots,x_r}$ such that $g(x_1,\dots,x_r)=g$ and $x_k\ge 1$ for each $k\ne i$.
By Theorem~\ref{thlie} the lemma follows.
\end{proof}

\begin{proof}[Proof of Theorem~\ref{thm-general}] By Lemma~\ref{lker}.a the sequence of Theorem~\ref{thhae} splits.
Hence there is an isomorphism $E^m_{\mathrm{B}}(\sqcup_{k=1}^r S^{p_k}) \otimes\mathbb{Q} \cong \kernel w$. Thus formula~(\ref{eq1}) 
follows from Lemma~\ref{lker}(b).
Formula~(\ref{eq1b}) follows from~(\ref{eq1}), Theorems~\ref{th-structure} and \ref{th1}, and
the computation
$m(t,0,\dots,0)=
\delta_{2t,5-(-1)^{m-p}}$.
\end{proof}

\begin{proof}[Proof of Theorem~\ref{th2}]
Let $S$ be a subset of $\{1,\dots,r\}$. Denote by $L(S)$ the subalgebra of $L$ generated by the generators $P_k$ such that $k\in S$. Then $L_g\cong\oplus_S L^0_g(S)$, where the sum is over all the subsets $S\subset\{1,\dots,r\}$, and $L_g\cong\oplus_{S\ni i} L^i_g(S)$.

By Lemma~\ref{lker}(a) and the proof of Theorem~\ref{thm-general} it follows that for a set $S$ with more than one element
$$
\rank E^m_{\mathrm{B}}(\sqcup_{k\in S} S^{p_k}) =
\sum_{k\in S}\dim L^{k}_{p_k-1}(S) - \dim L^{0}_{m-3}(S).
$$
Summing these formulas over all the nonempty subsets $S\subset\{1,\dots,r\}$ and using the result
$\rank E^m(S^{p_k}) =c_{p_k,m}$ \cite[Corollary~6.7]{Hae66A},
we get: 
$$
\rank E^m(\sqcup_{k=1}^r S^{p_k})=
\sum_{k=1}^{r}\left(\dim L_{p_k-1} - \dim L_{p_k-1}(\{k\})+\dim L_{m-3}(\{k\})+c_{p_k,m}\right)- \dim L_{m-3}.
$$
To see that this equals the expression from Theorem~\ref{th2}, use the formula $w_{t-1,m-p}(1)-w_{t,m-p}(1)=\delta_{2t,5-(-1)^{m-p}}$
and the following well-known result.
\end{proof}

\begin{theorem}[Witt's formula] \textup{(See \cite{Pet03})} Suppose that all $r$ generators of the free Lie superalgebra $L$ have the same degree $s$; then $\dim L_{ts}=w_{t,s}(r)$.
\end{theorem}

Corollary~\ref{cor-easy} follows easily from Theorem~\ref{th2}; see Section~\ref{ssec-fin-proofs} for the details.

The following claim will be used in the sequel.


\begin{claim}\label{cla8} If $r>2$ and $x_1=1$ then
$m(x_1,\dots,x_r)>0$. \end{claim}

\begin{proof} Analogously to Lemma~\ref{lker} we see that $m(x_1,\dots,x_r)$ is the dimension of the kernel of the restriction of the previous map
$$
w\colon L_{x_1-1,x_2,\dots,x_r}\oplus\dots\oplus L_{x_1,x_2,\dots,x_r-1}\to L_{x_1,x_2,\dots,x_r}.
$$
It suffices to prove that the map is not injective. Take products $u$ and $v$ of generators such that the element $[u,v]$ is nonzero and has multidegree $(0,x_2,\dots,x_r)$. The map $w$ is not injective because
$$
0=[P_1,[u,v]]\pm [u,[P_1,v]]\pm [v,[P_1,u]]=
[P_1,[u,v]]+[P_2,u_2]+\dots+[P_r,u_r]=w([u,v],u_2,\dots,u_r)
$$
for some choice of signs and elements $u_2,\dots,u_r\in L$. Here the first equality is the Jacobi identity. The second equality follows from Claim~\ref{cla5} because
$u$ and $v$ do not contain the generator $P_1$.
\end{proof}

\subsection{Finiteness criteria}\label{ssec-fin-proofs}

In this subsection we prove Theorems~\ref{thm-1a} and~\ref{thm-1b}. For this we solve the elementary problem of determining when formula~(\ref{eq4}) gives zero.
Formally, Theorems~\ref{thm-1a} and~\ref{thm-1b} follow from assertions~\ref{thm-general}, \ref{lalgb}, and~\ref{lalg}.


%

To explain the idea of the proof, let us prove Corollary~\ref{cor-easy} first.

\begin{proof}[Proof of Corollary~\ref{cor-easy}]
We have
\begin{multline*}
\rank E^{m}(S^{m-3}\sqcup S^{m-3})=
2w_{m-4,3}(2)-w_{m-3,3}(2)+2c_{m-3,m}-2\delta_{2m-6,6}\ge\\
2w_{m-4}(2)-w_{m-3}(2)-w_{(m-3)/2}(2)-2\ge
\frac{2^{m-3}}{(m-4)(m-3)}-\frac{2^{m/2}m}{m-3}-2\to\infty
\end{multline*}
as $m\to\infty$. Here the first equality follows from Theorem~\ref{th2}. The second inequality follows from the definitions in Section~\ref{ssec-ranks}. The third inequality follows from the observation that sum~(\ref{eq5}) contains at most $t$ summands, each being at most $2^{t/2}$ except the one corresponding to $i=1$.  \end{proof}

\begin{lemma}\label{lalgb}
For $r>2$ we have $m(x_1,\dots,x_r)>0$.
\end{lemma}

\begin{proof}[Proof of Lemma~\ref{lalgb}]
Without loss of generality assume that $x_1\le\dots\le x_r$.

{\it Case $x_1=1$} was proved in Claim~\ref{cla8} above.

{\it Case $x_1\ge 2$.} In this case we need a rather precise asymptotic estimate for the multiplicity $m(x_1,\dots,x_r)$. Denote by $n(x)$ the total number of those divisors of $x$ and $x-1$, which are free of squares and greater than $1$. Denote by $[x]$ the integral part of $x$.

\begin{claim}\label{cla2}
$m(x_1,\dots,x_r)\ge \frac{1}{(x_1+\dots+x_r)(x_1+\dots+x_r-1)}
\left(\binom{x_1+\dots+x_r}{x_1, \dots, x_r}-
n(x_1)\cdot(x_1+\dots+x_r)
\binom{[(x_1+\dots+x_r)/2]}{[x_1/2],\dots,[x_r/2]}\right).$
\end{claim}

\begin{proof}
For clarity set $r=2$; the general case is analogous. Let us estimate the number $m(x,y)=d(x-1,y)+d(x,y-1)-d(x,y)$. Substitute the given $3$ terms by the expressions from formula~(\ref{eq3}) in Theorem~\ref{thm-general}.
We obtain an algebraic sum of binomial coefficients divided by either $(x+y-1)$ or $(x+y)$:
$$
m(x,y)=\frac{\binom{x+y-1}{x-1}}{x+y-1}+\frac{\binom{x+y-1}{x}}{x+y-1}-\frac{\binom{x+y}{x}}{x+y}\pm\dots
$$
The algebraic sum of the first $3$ terms in this formula is equal to $\binom{x+y}{x}/(x+y-1)(x+y)$. Consider the remaining terms.
Their number is at most $n(x)$.
Let us estimate each one. By geometric interpretation of binomial coefficients it follows that $\binom{[(x+y)/k]}{[x/k]}\le\binom{[(x+y)/2]}{[x/2]}$
for $k\ge2$. Thus the value $\binom{[(x+y)/{2}]}{[{x}/{2}]}/(x+y-1)$ is not less than the absolute value of each term.
Thus Claim~\ref{cla2} follows.
\end{proof}

\begin{claim}\label{cla6} For $2\le x_1\le\dots \le x_r$ we have
$$
\left|\frac{m(x_1,\dots,x_r)}{M(x_1,\dots,x_r)}-1\right|<\epsilon(x_1,r),
$$
where $M(x_1,\dots,x_r)=\frac{(x_1+\dots+x_r-2)!}{x_1!\cdots x_r!}$
and
$\epsilon(x_1,r)=n(x_1)x_1{2}^{(r-1)/2}e^{r/9x_1}r^{1-rx_1/2}$.
\end{claim}

\begin{proof}
This follows from the estimates:
\begin{multline*}
\left|\frac{m(x_1,\dots,x_r)}{M(x_1,\dots,x_r)}-1\right| < n(x_1)(x_1+\dots+x_r)\binom{[{(x_1+\dots+x_r)}/{2}]}{[{x_1}/{2}],\dots,[{x_r}/{2}]}\binom{x_1+\dots+x_r}{x_1,\dots,x_r}^{-1}
< \\
n(x_1)(x_1+\dots+x_r){2}^{(r-1)/2}e^{r/9x_1}\prod_{i=1}^{r}\left(\frac{x_i}{x_1+\dots+x_r}\right)^{x_i/2} <n(x_1)x_1{2}^{(r-1)/2}e^{r/9x_1}r^{1-rx_1/2}=\epsilon(x_1,r).
\end{multline*}

Here the first inequality follows from Claim~\ref{cla2}.

The second inequality follows from the Stirling formula $x!=\sqrt{2\pi x}\left({x}/{e}\right)^x e^{\theta(x)/12x}$, where $0<\theta(x)<1$.

The third inequality follows from the assumption $2\le x_1\le \dots\le x_r$ and the monotonicity of the third expression in each variable $x_2,\dots,x_r$. To check the monotonicity, denote the expression by $f$. Then
$$
\frac{\partial\ln f}{\partial x_k}=
\frac{1}{x_1+\dots+x_r}-\frac{1}{2}\ln\frac{x_1+\dots+x_r}{x_k}\le
\frac{1}{2+x_k}-\frac{1}{2}\ln(1+\frac{2}{x_k})<0.
$$
\end{proof}

\begin{claim}\label{cla3} For $x>1$ we have $n(x)\le 2\sqrt{x}$. 
\end{claim}

\begin{proof}
For each divisor $d$ of $x$ we have either $d\le \sqrt{x}$ or $x/d\le\sqrt{x}$. Each integer $d=2,3,\dots,[\sqrt{x}]$ divides no more than one of the numbers $x$,
$x-1$. So these two numbers have at most $2\sqrt{x}+1$ positive divisors including $1$. The required inequality is proved.
\end{proof}

\begin{claim}\label{cla7} Assume that $r>2$ and $x_1>1$, or $r=2$ and $x_1>4$. Then $\epsilon(x_1,r)<1$.
\end{claim}

\begin{proof}
First note that $\epsilon(x,r)$ is a decreasing function in $r$ (for the values of the variables under consideration) because
$$
\frac{\partial \ln\epsilon(x,r)}{\partial r}=
\frac{1}{r}+\frac{1}{9x}+\frac{\ln 2}{2}-\frac{x}{2}(\ln r+1)<
\frac{1}{2}+\frac{1}{18}+\frac{\ln 2}{2}-\ln 2-1<0.
$$
So $\epsilon(x,r)\le \epsilon(x,3)<1$
for $x=2,3,4$. For $x>4$ the function $\epsilon(x,r)2\sqrt{x}/n(x)$ is similarly a decreasing function in~$x$. Thus by Claim~\ref{cla3} we have
$\epsilon(x,r)\le\epsilon(x,r)2\sqrt{x}/n(x)\le \epsilon(5,2)\sqrt{5}<1$.
\end{proof}

Now we can conclude the proof of Lemma~\ref{lalgb}.
By the assumptions $x_1>1$, $r>2$, and Claim~\ref{cla7} it follows that $\epsilon(x_1,r)<1$. Then by Claim~\ref{cla6} it follows that $m(x_1,\dots,x_r)>0$.
Lemma~\ref{lalgb} and hence Theorem~\ref{thm-1a} are proved.
\end{proof}

Let us proceed to $2$-component links.

\begin{lemma}\label{lalg} We have $m(x,y)>0$, if $(x,y)\in FCS (m-p_1,m-p_2)$, and $m(x,y)=0$, otherwise.
\end{lemma}

\begin{proof}[Proof of Lemma~\ref{lalg}]
%

Assume $x\le y$ without loss of generality. Consider cases $x=1,2,3,4$ and $x\ge 5$ separately.

{\it Case $x=1$} follows by a direct computation.

{\it Case $x=2$.} If $m-p_1$ and $m-p_2$ are even then by formula~(\ref{eq3}) in Theorem~\ref{thm-general}
we have $d(2,y)=[\frac{y+1}{2}]$. So $m(2,y)=d(1,y)+d(2,y-1)-d(2,y)\ne 0$ if and only if $2|y$, which is equivalent to $(2,y)\in FCS (1,1)$.
So for $x=2$, $m-p_1$ and $m-p_2$ even the lemma follows. The cases when either $m-p_1$ or $m-p_2$ is odd are considered analogously.

\begin{table}[h]
\caption{The ``multiplicities'' $m(x,y)$ of the points $(x,y)\in FCS (i,j)$ for small $x,y$.}
\label{dims}
\begin{center}
\begin{tabular}{ccc}
\begin{tabular}{|lr|lllll|}
\hline
\multicolumn{7}{|c|}{for $i,j$ even:}\\[2pt]
\hline
\multicolumn{2}{|l|}{5}& 0& 0& 1& 1& 3\\
\multicolumn{2}{|l|}{4}& 0& 1& 0& 2& 1\\
\multicolumn{2}{|l|}{3}& 0& 0& 1& 0& 1\\
\multicolumn{2}{|l|}{2}& 0& 1& 0& 1& 0\\
\multicolumn{2}{|l|}{1}& 1& 0& 0& 0& 0\\
\hline
$y$ &    &   &   &   &   & \\
    & $x$& 1 & 2 & 3 & 4 & 5 \\
\hline
\end{tabular}
\quad
&
\begin{tabular}{|lr|lllll|}
\hline
\multicolumn{7}{|c|}{for $i$ odd, $j$ even:}\\[2pt]
\hline
\multicolumn{2}{|l|}{5}& 0& 1& 1& 1& 3\\
\multicolumn{2}{|l|}{4}& 0& 0& 1& 2& 1\\
\multicolumn{2}{|l|}{3}& 0& 1& 1& 0& 1\\
\multicolumn{2}{|l|}{2}& 0& 0& 1& 1& 0\\
\multicolumn{2}{|l|}{1}& 1& 1& 0& 0& 0\\
\hline
$y$ &    &   &   &   &   & \\
    & $x$& 1 & 2 & 3 & 4 & 5 \\
\hline
\end{tabular}
\quad
&
\begin{tabular}{|lr|lllll|}
\hline
\multicolumn{7}{|c|}{for $i,j$ odd:}\\[2pt]
\hline
\multicolumn{2}{|l|}{5}& 0& 1& 1& 1& 3\\
\multicolumn{2}{|l|}{4}& 0& 0& 1& 2& 1\\
\multicolumn{2}{|l|}{3}& 0& 0& 1& 1& 1\\
\multicolumn{2}{|l|}{2}& 1& 1& 0& 0& 1\\
\multicolumn{2}{|l|}{1}& 1& 1& 0& 0& 0\\
\hline
$y$ &    &   &   &   &   & \\
    & $x$& 1 & 2 & 3 & 4 & 5 \\
\hline
\end{tabular}
\\[3pt]
\end{tabular}
\end{center}
\end{table}



{\it Case $x=3$.} For $y=3,4,5$ the lemma is proved by a direct computation, see Table~\ref{dims}. It remains to prove that for $y\ge 6$ we have $m(3,y)>0$. This follows from Claim~\ref{cla2} 
because for $x=3$ and $y\ge 6$ the right-hand side in Claim~\ref{cla2} is at least $\frac{y+1}{6}-\frac{y+3}{y+2}>0$.


{\it Case $x=4$.} For $y=4,5$ the lemma is proved by a direct computation. Let us prove that for $y\ge 6$
we have $m(4,y)>0$. This again follows from Claim~\ref{cla2} 
because for $x=4$ and $y\ge 6$
the right-hand side in Claim~\ref{cla2} is at least $\frac{(y+2)(y+1)}{24}-\frac{(y+4)(y+2)}{4(y+3)}>0$.

{\it Case $x\ge 5$.} For $y\ge x\ge 5$ we have $m(x,y)>0$ by Claims~\ref{cla6} and~\ref{cla7}.

Lemma~\ref{lalg} and hence Theorem~\ref{thm-1b} with Corollary~\ref{cor1} are proved.
\end{proof}

\subsection{Framed links}\label{ssec-framed-proofs}


The group structure, analogue of the decomposition from Theorem~\ref{th-structure}, and the following exact sequence for partially framed links are constructed completely analogously to the particular case of framed links ($l_k=m-p_k$) studied in \cite{Hae66A}.

\begin{theorem}\label{th4} \textup{(Cf. \cite[Corollary~5.9]{Hae66A})} For $p<m-2$ 
there is an exact sequence of finitely generated abelian groups
$$
\dots \to \pi_{p}(V_{m-p,l}) \xrightarrow{\tau} {E}^m(S^p\times D^l) \xrightarrow{i^*} E^{m}(S^p)
\xrightarrow{Ob} \pi_{p-1}(V_{m-p,l})\to {E}^{m-1}(S^{p-1}\times D^{l})\to\dots
$$
\end{theorem}



Theorem~\ref{th-framed} follows immediately from the following result:

\begin{lemma}\label{l-framedsplit}
After tensoring with $\mathbb{Q}$ the sequence of Theorem~\ref{th4} splits as
$$
0 \to \pi_{p}(V_{m-p,l}) \tensor \Q \xrightarrow{\tau\tensor\mathrm{Id}_\Q} {E}^m(S^p\times D^l) \tensor \Q \xrightarrow{i^*\tensor\mathrm{Id}_\Q} E^{m}(S^p) \tensor \Q  \to 0.
$$
\end{lemma}

For the proof of the lemma we need to compute the rational homotopy groups of orthogonal groups and describe their generators. Let us introduce some notation.
Let $\mathrm{Vect}_q(S^{p+1})$ be the set of isomorphism classes of oriented vector bundles of rank $q$ over $S^{p+1}$.  Viewing a representative of $x \in \pi_p(SO_q)$ as a clutching function for a vector bundle over $S^{q+1}$ gives the well-known bijection $c \colon \pi_p(SO_q) \cong {\rm Vect}_q(S^{p+1})$ which we use to identify ${\rm Vect}_q(S^{p+1})$ as an abelian group.  We use $c$, the \emph{Pontryagin class} and the \emph{Euler class} to define homomorphisms
\[ \widehat P, \widehat E \colon \pi_p(SO_q) \tensor \Q \cong \mathrm{Vect}_q(S^{p+1})  \tensor \Q \to
H^{p+1}(S^{p+1};\mathbb{Q})\cong \mathbb{Q}
\]
as follows: for $V \in {\rm Vect}_q(S^{p+1})$ and for $\alpha \in \Q$ define
\[ \widehat P(V \tensor \alpha) := \left\{ \begin{array}{cl} P_k(V) \tensor \alpha, & \text{if $p=4k-1$,} \\ 0 & \text{if $p \neq 4k-1$;} \end{array} \right. \quad \text{and} \quad \widehat E(V \tensor \alpha) := \left\{ \begin{array}{cl} E(V) \tensor \alpha & \text{if $(p, q)=(2j-1, 2j)$,} \\ 0 & \text{if $(p, q)\neq (2j-1, 2j),$} \end{array} \right. \]
where $P_k$ and $E$ denote respectively the $k$-th Pontryagin class and the Euler class.  We then define
$P_{p, q} =P_k :=\widehat{P}^{-1}(1)\in \pi_p(SO_q) \tensor \Q$, if the map $P\colon \pi_p(SO_q)\to \mathbb{Q}$ is nontrivial (and thus $p=4k-1$ for some integer $k$), and $P_{p,q}=0$, otherwise. We also define $E_{p,q} = E$ analogously.

\begin{claim}\label{cla9}
There are isomorphisms
\[
\pi_{p}(SO_q) \tensor \Q \cong \left \{ \begin{array}{cl} \Q(P_k) \oplus \Q(E) & \text{if $p = q-1 = 4k - 1$,} \\
\Q(P_k) &\text{if $p = 4k -1$ and $p \neq q-1 > p/2$,}\\ \Q(E) & \text{if $p = q-1 = 4k+1$,} \\
0 & \text{otherwise.}
\end{array} \right.
\]
\end{claim}

\begin{proof}[Proof of Claim \ref{cla9}]
This is well-known: we give a proof using results in \cite{F-H-T01}.   A general result on the rational homotopy type of  homology-finite $H$-spaces \cite[\S 12(a) Example 3]{F-H-T01} ensures that $\pi_*(SO_q) \tensor \Q$ has generators only in odd dimensions.  In \cite[\S 15(f) Example 3]{F-H-T01} these generators are given:
\[ SO_{2n+1} : \{x_1, \dots x_n \}, \]
\[ SO_{2n} : \{ x_1, \dots, x_{n-1}, x_n' \}. \]
Here $x_k$ has dimension $4k-1$ and $x'_n$ has dimension $2n-1$.  Moreover the generators $x_k$ are stably nontrivial and so are detected by the $k$-th Pontryagin class and the generator $x_n'$ is detected by the Euler class.  This proves the claim.
\end{proof}



\begin{proof}[Proof of Lemma~\ref{lstiefel}]
Consider the exact homotopy sequence of the fibration $SO_{q-l}\to SO_{q}\to V_{q,l}$ defined by forgetting the first $q-l$ vectors of an oriented $q$-frame in $\R^q$:
$$
\dots \to
\pi_p(SO_{q-l})     \xrightarrow{i_p}
\pi_p(SO_q)         \xrightarrow{j_p}
\pi_p(V_{q,l})      \xrightarrow{\partial_p} \pi_{p-1}(SO_{q-l}) \xrightarrow{i_{p-1}} \pi_{p-1}(SO_q)
\to \dots
$$
By exactness we have that
$\pi_p(V_{q,l})\otimes \mathbb{Q}\cong
\mathrm{coker}(i_p\otimes\mathrm{Id}_\mathbb{Q})
\oplus \mathrm{ker}(i_{p-1}\otimes\mathrm{Id}_\mathbb{Q})$.

Henceforth we consider all groups and homomorphisms to be tensored with $\Q$, respectively ${\rm Id}_\Q$, however we omit $\otimes\mathbb{Q}$ and $\otimes{\rm Id}_\Q$ from the notation of groups and homomorphisms.  By Claim~\ref{cla9} and the exact sequence above, the group $\pi_p(V_{q,l})$ is generated by those elements among $j_p E_{p,q}$, $j_p P_{p,q}$, $\partial_p^{-1} E_{p-1,q-l}$, $\partial_p^{-1}P_{p-1,q-l}$ which exist and are non-zero. It remains to investigate when each of the elements vanishes.

Since the Euler class vanishes for bundles with section see that $j_p E_{p,q}\ne 0$ iff $E_{p,q}\ne0$ and that $i_{p-1}E_{p-1, q-1} = 0$.  We deduce that $j_p E_{p,q}\ne 0$ iff $2\,|\,p+1=q$ and that  $\partial_p^{-1} E_{p-1,q-l}\ne 0$ iff $2\,|\,p=q-l$.

Moving to the classes $P_{p,q}$, observe that if both $P_{p, q-l}$ and $P_{p,q}$ are nonzero then $i_p P_{p,q-l}=P_{p,q}$ because the Pontryagin classes are stable.  Hence $j_p P_{p,q}\ne 0$ iff $P_{p,q}\ne0$ and $P_{p,q-l}=0$. Applying Claim~\ref{cla9}
we deduce that $j_p P_{p,q}\ne 0$ iff $4\,|\,p+1$ and $p/2+1<q<l+p/2+1$.

Finally, $\partial_p^{-1} P_{p-1,q-l}\ne 0$ iff
$P_{p-1,q-l}\ne 0$ and $P_{p-1,q}=0$.  By stability we see that $\partial_p^{-1} P_{p-1,q-l}= 0$.

Thus $\rank \pi_p(V_{q,l})=2$, if both $j_p E_{p,q}, j_p P_{p,q}\ne0$. We have $\rank \pi_p(V_{q,l})=1$, if exactly one of the elements $j_p E_{p,q}$, $j_p P_{p,q}$, $\partial_p^{-1} E_{p-1,q-l}$ exists and is nonzero. Otherwise $\rank \pi_p(V_{q,l})=0$.  Combining the conditions for these cases we found above, we obtain the statement of the lemma.
\end{proof}

\begin{proof}
[Proof of Lemma~\ref{l-framedsplit}]
It suffices to prove that the map $E^{m}(S^p)\to \pi_{p-1}(V_{m-p,l})$ from the sequence of Theorem~\ref{th4} has finite image for $p<m-2$.  If either $p+1$ is not divisible by $4$ or $m\ge 3p/2+2$ then by Theorem~\ref{th1} the group $E^{m}(S^p)$ is finite. If $p+1$ is divisible by $4$ and $m<3p/2+2$ then by Lemma~\ref{lstiefel} the group $\pi_{q-1}(V_{m-q,p})$ is finite. The lemma now follows.
\end{proof}

\section{Applications and related problems} \label{sec-further}
In this section we discuss applications of our rational calculations of the groups of links to the classification of handlebodies and thickenings and to the computation of certain mapping class groups.  We also raise some open problems. Throughout this section denote $q_k := m-p_k-1$.

\subsection{Handlebodies and thickenings of wedges of spheres} \label{ssec-handle}

Recall that an $(m+1)$-dimensional handlebody $V$ is a connected manifold with boundary which is obtained from the disc $D^{m+1}$ by attaching a finite set of handles:
\[ V \cong D^{m+1} \cup_f \bigl(\sqcup_{k=1}^r D^{p_k+1} \times D^{q_k + 1} \bigr) \]
where $f \colon \sqcup_{k=1}^r S^{p_k} \times D^{q_k+1} \to S^m = \partial D^{m+1}$ is a framed link.  We define
\[ \mathcal{H}_+(m+1; p_1+1, \dots, p_r+1) : = \{ V \}/\text{o.\,p.\,diffeomorphism} \]
to be the set of oriented diffeomorphism classes of handlebodies with precisely $r$ handles of respective dimensions $p_1 + 1, \dots, p_r + 1$.

Determining the set $\mathcal{H}_+(m; p_1+1, \dots p_r+1)$ can be an intricate problem and, as is often the case, it helps to introduce extra structure.  The manifold $V$ above comes equipped with an obvious homotopy equivalence $\phi \colon \vee_{k=1}^r S^{p_k +1} \simeq V$.  Following \cite{Wa66} and \cite[\S 6]{Hae66C} one defines the set of equivalence classes of homotopy equivalences
\[ \mathcal{T}^{m+1}(\vee_{k=1}^r S^{p_k + 1}) = \{ \phi \colon \! \vee_{k=1}^r S^{p_k + 1} \simeq W \}/ \sim \]
where $W$ is a connected oriented smooth manifold with a non-empty simply connected boundary and $(W_0, \phi_0) \sim (W_1, \phi_1)$ if and only if there is an orientation preserving diffeomorphism $h \colon W_0 \cong W_1$ such that $\phi_1$ is homotopic to $h \circ \phi_0$.  The group of homotopy classes of self-homotopy equivalences of $\vee_{k=1}^r S^{p_k + 1}$, denoted  $\mathcal{E}(\vee_{k=1}^r S^{p_k + 1})$, acts by pre-composition on $\mathcal{T}^{m+1}(\vee_{k=1}^r S^{p_k + 1})$ and the obvious forgetful map defines a bijection
\[   \mathcal{T}^{m+1}(\vee_{k=1}^r S^{p_k + 1}) / \mathcal{E}(\vee_{k = 1}^r S^{p_k + 1}) \equiv \mathcal{H}_+(m+1; p_1+1, \dots, p_r+1), \quad [W, \phi] \mapsto [W]. \]

As described above attaching handles defines a map
\[ \omega \colon E^m(\sqcup_{k=1}^rS^{p_k} \times D^{q_k + 1}) \to \mathcal{T}^{m+1}(\vee_{k=1}^r S^{p_k + 1}) \]
and Haefliger proved the following:
\begin{theorem}\textup{See \cite[Theorem 6.2]{Hae66C}}
The map $\omega \colon E^m(\sqcup_{k=1}^r S^{p_k} \times D^{q_k + 1}) \to \mathcal{T}^{m+1}(\vee_{k=1}^r S^{p_k+1})$ is surjective if $2p_k - p_j + 2 \leq m$ and $2p_k - p_j \geq 1$ and injective if $2p_k - p_j +2 < m$ and $2p_k - p_j > 1$ for all $k, j = 1, \dots, r$.
\end{theorem}

\noindent
As a consequence of our calculations of $E^m(\sqcup_{k=1}^r S^{p_k} \times D^{q_k + 1}) \tensor \Q$ we immediately obtain the following:

\begin{corollary} \label{cor-thickfinite1}
Suppose that $2p_k - p_j + 2 \leq m$ and $2p_k - p_j \geq 1$ for all $k, j = 1, \dots, r$.  If the multi-index $(p_1, \dots, p_r)$ fails to satisfy each of the conditions of Corollary~\ref{cor-globalframed}
then both of the sets $\mathcal{T}^{m+1}(\vee_{k=1}^r S^{p_k+1})$ and $\mathcal{H}_+(m+1; p_1+1, \dots, p_r + 1)$ are finite.
\end{corollary}

\begin{corollary} \label{cor-thickfinit2}
Suppose that $2p_k - p_j +2 < m$ and $2p_k - p_j > 1$ for all $k, j = 1, \dots, r$, then $\mathcal{T}^{m+1}(\vee_{k=1}^r S^{p_k + 1})$ inherits the structure of a finitely generated abelian group via $\omega$ and the rank of this group equals to the right-hand side
of the equality from Theorem~\ref{th-framed}.
\end{corollary}

\begin{example}
By Corollary~\ref{cor-globalframed}, $E^8(S^5 \times D^3 \sqcup S^5 \times D^3)$ is finite and so both the group $\mathcal{T}^9(S^6 \vee S^6)$ and the set $\mathcal{H}_+(9; 6,6)$ are finite.
\end{example}

\begin{remark}
Computing the action of the group $\mathcal{E}(\vee_{k = 1}^r S^{p_k + 1})$ on the set $\mathcal{T}^{m+1}(\vee_{k = 1}^r S^{p_k +1})$ is an interesting problem we shall address in \cite{CFS11}.  In particular, we conjecture that the converse of the second sentence of
Corollary~\ref{cor-thickfinite1} holds under the dimension assumptions of the first sentence.
\end{remark}

\subsection{The mapping class group
of a connected sum of ``tori''} \label{ssec-mapping}
Recall that the group of framed links, $E^m(\sqcup_{k=1}^r S^{p_k} \times D^{q_k + 1})$, consists of smooth isotopy classes of smooth orientation preserving embeddings
\[ f \colon \! \sqcup_{k=1}^r S^{p_k} \times D^{q_k + 1} \to S^m.\]
In this subsection we describe the relationship of $E^m(\sqcup_{k=1}^r S^{p_k} \times D^{q_k+1})$ to certain mapping class groups.

We begin with some necessary new definitions.  Let us first expand the definition of link to include links in homotopy spheres.  Let $\Sigma^m$ be a closed smooth oriented homotopy $m$-sphere and recall that $\Theta_m$ denotes the group of oriented diffeomorphism classes of such homotopy spheres.  Define
\[ E^{\Theta_m}(\sqcup_{k=1}^r S^{p_k} \times D^{q_k + 1}) := \{ f \colon \! \sqcup_{k=1}^r S^{p_k} \times D^{q_k + 1} \to \Sigma^m \}/ \simeq \]
to be the set of equivalence classes of orientation preserving embeddings of $\sqcup_{k=1}^r S^{p_k} \times D^{q_k + 1}$ into some homotopy sphere where embeddings $f_0$ and $f_1$ with targets $\Sigma^m_0$ and $\Sigma^m_1$ respectively are equivalent if there is an orientation preserving diffeomorphism $H \colon \Sigma^m_0 \cong \Sigma^m_1$ such that $H \circ f_0 = f_1$.

The isotopy extension theorem ensures that there is a well-defined map
%
\[ I \colon E^{m}(\sqcup_{k=1}^r S^{p_k} \times D^{q_k + 1}) \to E^{\Theta_m}(\sqcup_{k=1}^r S^{p_k} \times D^{q_k + 1}). \]
Moreover there is a group structure on $E^{\Theta_m}(\sqcup_{k=1}^r S^{p_k} \times D^{q_k + 1})$ defined in the same was as the group structure on $E^m(\sqcup_{k=1}^r S^{p_k} \times D^{q_k + 1})$ but also taking connected sum of the ambient homotopy sphere.  In particular $I$ is a homomorphism and forgetting the embedding defines a homomorphism ${\rm Forg} \colon E^{\Theta_m}(\sqcup_{k=1}^r S^{p_k} \times D^{q_k + 1}) \to \Theta_m$.  Moreover ${\rm Forg}$ is split by the homorphism $T \colon \Theta_m \to E^{\Theta_m}(\sqcup_{k=1}^r S^{p_k} \times D^{q_k + 1})$ defined by mapping a homotopy sphere $\Sigma^m$ to the trivial framed link in $\Sigma^m$, i.e., the framed link which bounds disjoint embedded discs.

\begin{lemma} \label{lem:embext}
The sum of the homomorphisms $I$ and $T$ defines an isomorphism
\begin{equation} \label{seq:framed}
I \oplus T \colon  E^m(\sqcup_{k=1}^r S^{p_k} \times D^{q_k + 1}) \oplus \Theta_m \,\, \cong \,\,  E^{\Theta_m}(\sqcup_{k=1}^r S^{p_k} \times D^{q_k + 1}).
\end{equation}
\end{lemma}

\begin{proof}
It remains only to prove that $I$ is injective.  Suppose that $f_0, f_1 \colon \! \sqcup_{k=1}^r S^{p_k} \times D^{q_k + 1} \to S^m$ are embeddings such $I([f_0]) = I([f_1])$.  This means that there is an orientation preserving diffeomorphism $H \colon S^m \to S^m$ such that $f_1 = f_0 \circ H$.  If in addition $H$ were isotopic to the identity we would be done.  But since $H$ is orientation preserving we may assume after a small isotopy that it is the identity on on a small disc $D^m \subset {\rm Int}(S^m - f_0(\sqcup_{k=1}^r S^{p_k} \times D^{q_k + 1}))$.  We may now modify $H$ by any diffeomorphism of $D^m$ which is the identity on a neighbourhood of the boundary $\partial D^m$, keeping the diffeomorphisms the same outside of this disc.  Now it is well-know that every orientation preserving diffeomorphism of $S^m$ is pseudo-isotopic to a diffeomorphism which point-wise fixes some $m$-disc; see \cite[Hilfsatz\,p.265]{Wa63}.  In this way we can construct a diffeomorphism $H'$ such that $f_1 = f_0 \circ H'$ and such that $H'$ is isotopic to the identity.
\end{proof}

We also need to pass from disconnected manifolds to connected manifolds.  Hence we define
\[ W_0 : =  \natural_{k=1}^r (S^{p_k} \times D^{q_k +1}), \quad W_1 : = \natural_{k=1}^r (D^{p_k+1} \times S^{q_k} ) \quad \text{and} \quad M := \sharp_{k=1}^r(S^{p_k} \times S^{q_k}).\]
Note that there are canonical identifications $\partial W_0 = M = \partial W_1$.  We recall how the standing assumption of high-codimension gives rise to a bijection
\begin{equation} \label{eq:alpha}
 \alpha \colon E^{\Theta_m}(\sqcup_{k=1}^r S^{p_k} \times D^{q_k + 1}) \equiv E^{\Theta_m}(W_0)
\end{equation}
where $E^{\Theta_m}(W_0)$ is the set of isotopy classes of orientation preserving embeddings of $W_0$ into some homotopy $m$-sphere.  Given an embedding $f \colon \! \sqcup_{k=1}^r S^{p_k} \times D^{p_k + 1} \to \Sigma^m$ as above, choose smoothly embedded paths between each component of $f$ and remove a small tubular neighbourhood of each path: since each component of $f$ has codimension at least $3$ the choice of such a path does not play a role up to isotopy.  We obtain an embedding
\[ F \colon W_0 \to \Sigma^m. \]
Moreover, two such embeddings $f$ and $f'$ are isotopic if and only if the corresponding embeddings $F$ and $F'$ of $W_0$ are isotopic.  Conversely, given an embedding $F \colon W_0 \to \Sigma_m$ we may delete a set of appropriate $m$-discs from $W_0$ to obtain an embedding $f \colon\! \sqcup_{k=1}^r S^{p+k} \times D^{q_k + 1} \to \Sigma^m$.  In this way we obtain the the bijection $\alpha$ above.

We now relate $E^{\Theta_m}(W_0)$ to the mapping class group $\wt{\pi}_0\Diff_+(M)$ of pseudo-isotopy classes of orientation preserving diffeomorphisms of $M$.  Given $[h] \in \wt{\pi}_0\Diff_+(M)$ form the oriented smooth homotopy sphere
\[ \Sigma_{[h]} := W_0 \cup_{h} W_1 \]
where we orient $\Sigma_{[h]}$ using a fixed orientation on $W_0$.
%
%
Now define the map
\[ \beta \colon \wt{\pi}_0 \Diff_+(M)  \to E^{\Theta_m}(W_0), \quad [h] \mapsto  (W_0 \to \Sigma_{[h]}) \]
which maps $[h]$ to the obvious embedding of $W_0$ into $\Sigma_{[h]}$.  We shall see that if $h$ extends over $W_1$ then $\Sigma_{[h]} \cong S^m$ and $\beta([h])$ is the standard embedding.  Hence we introduce the homomorphisms defined by restriction of diffeomorphisms:
\[ R \colon \wt{\pi}_0\Diff_+(W_1) \to \wt{\pi}_0\Diff_+(M) .\]

\begin{lemma} \label{prop:mcg}
Suppose that $p_k \geq [m/2]$ for each $k = 1, \dots, r$.  Then the map $\beta$ above induces a bijection
\[ \bar \beta \colon R(\wt{\pi}_0\Diff_+(W_1)) \slash \wt{\pi}_0\Diff_+(M) \equiv E^{\Theta_m}(W_0).\]
\end{lemma}

\begin{corollary} \label{cor:mcg}
If $m \geq 5$, the subgroup $R(\wt{\pi}_0\Diff_+(W_1)) \subset \wt{\pi}_0\Diff_+(M)$ is of finite index if and only if $(p_1, \dots, p_k)$ fails to satisfy each of the conditions of Corollary~\ref{cor-globalframed}.
\end{corollary}

\begin{proof}
Let us prove the ``if'' part; the ``only if'' part is analogous. By~\cite{K-M63} the group $\Theta_m$ is finite and by assumption Corollary~\ref{cor-globalframed} states that $E^m(\sqcup_{k=1}^r S^{p_k} \times D^{q_k+1})$ is finite.  By ~\eqref{seq:framed} and ~\eqref{eq:alpha} we see that $E^{\Theta_m}(W_0)$ is finite.  Now apply Lemma~\ref{prop:mcg}.
\end{proof}

\begin{proof}[Proof of Lemma \ref{prop:mcg}]
We first show that $\bar \beta$ is well-defined.  Let $[h_a]$ and $[h_b]$ be elements of $\wt{\pi}_0\Diff(M)$ and suppose that  $h_b = h \circ h_a$ where $h \colon M \cong M$ extends to a diffeomorphism $H \colon W_1 \cong W_1$.  Then have the diffeomorphism
\[ {\rm Id} \cup H \colon W_0 \cup_{h_a} W_1 \to W_0 \cup_{h_b} W_1 \]
which shows that $\beta([h_a]) = \beta([h_b])$.  It follows that each member of a coset of ${\rm Im}(R)$ is mapped to the same isotopy class of embedding via $\beta$ and hence $\bar \beta$ is well-defined.

Next we show that $\bar \beta$ is injective: suppose that $\bar \beta([h_a]) = \bar \beta([h_b])$.  Then there is a diffeomorphism
\[ G \colon W_0 \cup_{h_a} W_1 \cong W_0 \cup_{h_b} W_1 \]
such that $G|_{W_0} = {\rm Id}$.  It follows that $G(W_1) = W_1$ and that for $h : = R(G|_{W_1})$, $h_b = h \circ h_a$.  It follows that $[h_a]$ and $[h_b]$ lie in the same coset of $R(\wt{\pi}_0\Diff(W_1))$ and so $\bar \beta$ is injective.

Finally we show that $\bar \beta$ is surjective.  For any embedding $F \colon W_0 \to \Sigma^m$ let $C_F \subset \Sigma^m$ denote the closure of its complement.  Then $C_F$ is homotopy equivalent to the wedge $\vee_{k=1}^r S^{q_k}$.  Since $C_F$ is simply connected with non-empty simply-connected boundary it is the target of a thickening in the sense of \cite{Wa66}.  Moreover, as $p_k \geq [m/2]$ it follows that $q_k < [m/2]$ and thus $C_F$ is the target of a stable thickening of $\vee_{k=1}^r S^{q_k}$.  Now by \cite[Proposition 5.1]{Wa66} stable thickenings are classified by their stable tangent bundles, $C_F \subset \Sigma^m$ and $\Sigma^m$ is stably parallizable: it follows that $C_F$ is the target of the trivial thickening.  Hence for every such embedding $F$ there is a diffeomorphism
\[ H \colon C_F \cong W_1. \]
We define the diffeomorphism $h := (\partial H) \circ (\partial F) \colon M \cong M$.  It follows that the embedding $F \colon W_0 \to \Sigma$ is isotopic to the embedding $\bar \beta(h)$ and we are done.
%
\end{proof}

\subsection{Open problems} \label{ssec-open}

In this subsection we discuss some open problems related to the groups of links.  The first of these is the computation of the torsion group ${\rm Tors}\,E^m(\sqcup_{k=1}^r S^{p_k})$ which by the Haefliger sequence \cite[Theorem~1.1]{Hae66C} is a finite abelian group. Here we wish to point out that the analysis of the Haefliger sequence seems significantly more difficult than analysing the sequence tensored with $\Q$.  Specifically, our computation of $E^m(\sqcup_{k=1}^r S^{p_k}) \tensor \Q$ rested on the Lemma~\ref{lker}(a) which stated that the homomorphism $w$ of the exact sequence in Theorem \ref{thhae} is surjective.  This implied that the rational sequence of Theorem~\ref{thhae} splits into short exact sequences and also the rational splitting in Lemma \ref{lem:Hae-splits}.  However the integral Haefliger long exact sequence does split in this way in general.

\begin{lemma} \label{lem-integralnonsplit}
The homomorphism
$w \colon \! {}^0\!\Lambda_{(5, 5)}^{(2, 2)} \to {}^0\Pi_{6}^{(2, 2)}$ is not onto and so $0\ne\mu({}^0\Pi_6^{(2, 2)}) \subset E^7_{\rm B}(S^4 \sqcup S^4)$.
\end{lemma}

\begin{proof}
By the Hilton--Milnor theorem and known computations of the homotopy groups of spheres it follows that the $3$-torsion in ${}^0\Pi_6^{(2, 2)}$ is isomorphic to $\Z/3$ whereas the torsion in ${}^0\!\Lambda_{(5, 5)}^{(2, 2)}$ is all $2$-primary.  Moreover, by Example \ref{rem-E558}, the homomorphism $w\tensor {\rm Id}_\Q \colon {}^0\!\Lambda_{(5, 5)}^{(2, 2)} \tensor \Q \to {}^0\Pi_{6}^{(2, 2)} \tensor \Q$ is an isomorphism.  It follows that the $3$-torsion in ${}^0\Pi_6^{(2, 2)}$ is not in the image of $w$.
\end{proof}




Of course, to understand a finite abelian group a key first step is to find the primes which appear in its order.  We therefore post the following

\begin{problem}
Find a function of two positive integer variables $f(x, y)$ such that if a prime $q$ satisfies the inequality $q > f(m, {\rm min}\{ p_1, \dots, p_r \})$ then the group $E^m(\sqcup_{k=1}^r S^{p_k})$ contains no $q$-torsion.
\end{problem}

We now briefly consider the problem of classifying links whose components need not be spheres.  Let $N_1, \dots, N_r$ be a set of closed connected smooth manifolds with $N_k$ of dimension $p_k$ and consider the map
\[  C \colon E^m(\sqcup_{k=1}^r N_k) \to \sqcup_{k=1}^r E^m(N_k) \]
which associates to each isotopy class of embedding the set of isotopy classes of its components.   For a survey of what is known about $E^m(N_k)$ see \cite{Sko07L} and for the special cases, $(m; p_k) = (6, 3)$ or $(7, 4)$ see \cite{Sko07F,Sko08Z,CrSk08}.  If $E^m(N_k)$ is known for each $N_k$ then for a fixed isotopy classes $[f_k] \in E^m(N_k)$, $k = 1, \dots, r$, the set $C^{-1}([f_1], \dots, [f_r])$ becomes of particular interest.


In the case of spheres recall that the links with unknotted components are defined by the equality
\[ E^m_{\mathrm{U}}(\sqcup_{k=1}^r S^{p_k}) = C^{-1}(0, \dots, 0) \]
where $0 \in E^m(S^{p_k})$ denotes the class of the unknot.  Now the operation of component-wise connected sum defines a group action
\[  E^m_{\mathrm{U}}(\sqcup_{k=1}^r S^{p_k}) \times C^{-1}([f_1], \dots , [f_r]) \to C^{-1}([f_1], \dots , [f_r]). \]

\begin{problem}
Find a range of dimensions $(m; p_1, \dots, p_r)$ and connectivity assumptions on $N_k$ for which the above action is transitive.
\end{problem}


Finally we comment that there is a lack of explicit constructions of high codimension links. There is a construction of (the image of) the \emph{Borromean rings} $S^{2k-1}\sqcup S^{2k-1}\sqcup S^{2k-1}\to \mathbb{R}^{3k}$:
\begin{equation*}
2|\mathbf{x}|^2+|\mathbf{y}|^2=1,
\mathbf{z}=0;
\quad\sqcup\quad
2|\mathbf{y}|^2+|\mathbf{z}|^2=1,
\mathbf{x}=0;
\quad\sqcup\quad
2|\mathbf{z}|^2+|\mathbf{x}|^2=1,
\mathbf{y}=0;
\end{equation*}
where $\mathbf{x}=(x_1,\dots,x_k)$,
$\mathbf{y}=(x_{k+1},\dots,x_{2k})$,
$\mathbf{z}=(x_{2k+1},\dots,x_{3k})$ and $| \cdot |$ denotes the usual norm in~$\R^{k}$.
The Borromean rings generate the $\Q$-linear space $E^{3k}_{\mathrm{B}}(\sqcup_{k=1}^3 S^{2k-1})\otimes\Q\cong\Q$.
It would be interesting to obtain a shorter proof of Theorem~\ref{thm-1a} using either a generalization of this explicit construction
and the Milnor invariants or identities in Lie algebras (cf.~the proof of Lemma~\ref{lalgb} for the case when $x_1=1$).

It would be also interesting to determine when the set of isotopy classes of (Brunnian) links is finite without the codimension $>2$ assumption.



\subsection*{Acknowledgements}
We would like to thank P.~Akhmetiev, I.~Arzhantsev, S.~Duzhin, A.~Kanel, A.~Klyachko, S.~Melikhov, V.~Nezhinsky, V.~Petrogradsky, K.~Orr and A.~Skopenkov for helpful discussions.  The first two authors would like to thank Shmuel Weinberger and the University of Chicago for hosting them during the winter term of 2003.



\bibliographystyle{amsplain}


\noindent
\textsc{Diarmuid Crowley\\
Hausdorff Research Institute for Mathematics, \\
Poppelsdorfer Allee 82, \\
D-53115 Bonn, Germany}\\
\texttt{diarmuidc23@gmail.com} \quad \url{http://www.dcrowley.net}

\vskip 0.5cm
\noindent
\textsc{Steven C. Ferry\\
Department of Mathematics, Rutgers University,\\
Hill Center, Busch Campus, \\
Piscataway, NJ 08854-8019, USA} \\
\texttt{sferry@math.rutgers.edu} \quad \url{http://www.math.rutgers.edu/~sferry/}

\vskip 0.5cm
\noindent
\textsc{Mikhail Skopenkov\\
King Abdullah University of Science and Technology, \\
4700 Thuwal, 23955-6900, Kingdom of Saudi Arabia \\
and\\
Institute for information transmission problems of the Russian Academy of Sciences, \\
Bolshoy Karetny per. 19, bld. 1, Moscow, 127994, Russian Federation}\\
\texttt{skopenkov@rambler.ru} \quad \url{http://skopenkov.ru}

\end{document}